\documentclass[letterpaper,12pt]{article}

\usepackage{palatino}

\usepackage{amsmath,xcolor}
\usepackage{hyperref}
\definecolor{ao(english)}{rgb}{0.0, 0.5, 0.0}
\hypersetup{
	colorlinks=true,
	linkcolor=ao(english),
	citecolor=ao(english),
	filecolor=magenta,      
	urlcolor=cyan,
	linktoc=page
}

\usepackage{amsthm}
\usepackage{amssymb}
\usepackage{enumitem}
\usepackage{placeins}
\usepackage{graphicx,caption}
\usepackage{epstopdf}
\usepackage{xcolor}
\usepackage[longnamesfirst]{natbib}
\bibpunct[ ]{(}{)}{,}{a}{}{,}

\theoremstyle{plain}
\newtheorem{theorem}{Theorem}[section]
\newtheorem{corollary}[theorem]{Corollary}
 
\newtheorem{prop}[theorem]{Proposition}
\newtheorem{lemma}[theorem]{Lemma}

\theoremstyle{definition}
\newtheorem{definition}[theorem]{Definition}
\newtheorem{remark}[theorem]{Remark}
\newtheorem{example}[theorem]{Example} 

\newcommand{\Z}{\mathbb{Z}}

\newcommand{\N}{\mathbb{N}}
\newcommand{\R}{\mathbb{R}}
\newcommand{\p}{\mathbb{P}}

\newcommand{\fl}[1]{{\left\lfloor #1 \right\rfloor}}
\newcommand{\cl}[1]{{\left\lceil #1 \right\rceil}}
\newcommand{\sset}{\subset}

\newcommand{\mathand}{\;\text{and}\;}

\newcommand{\Ga}{\Gamma}
\newcommand{\ep}{\epsilon}

\newcommand{\de}{\delta}

\newcommand{\sig}{\sigma}
\newcommand{\Sig}{\Sigma}
\newcommand{\eps}{\epsilon}
\newcommand{\la}{\lambda}
\newcommand{\al}{\alpha}

\newcommand{\sD}{\mathcal{D}}
\newcommand{\sE}{\mathcal{E}}

\newcommand{\sG}{\mathcal{G}}

\usepackage{MnSymbol}

\newcommand{\eqd}{\stackrel{d}{=}}

\newcommand{\X}{\times}

\newcommand{\cvgdown}{\downarrow}

\newcommand{\bw}{\mathbf{w}}
\newcommand{\bx}{\mathbf{x}}
\newcommand{\by}{\mathbf{y}}
\newcommand{\bz}{\mathbf{z}}
\newcommand{\bp}{\mathbf{p}}
\newcommand{\bq}{\mathbf{q}}

\newcommand{\geole}{\trianglelefteq}

\title{RSK in last passage percolation: a unified approach}
\author{Duncan Dauvergne \and Mihai Nica \and B\'alint Vir\'ag}
\begin{document}
\maketitle
\begin{abstract}
We present a version of the RSK correspondence based on the Pitman transform and geometric considerations. This version unifies ordinary RSK, dual RSK and continuous RSK. We show that this version is both a bijection and an isometry, two crucial properties for taking limits of last passage percolation models.  

We use the bijective property to give a non-computational proof that dual RSK maps Bernoulli walks to nonintersecting Bernoulli walks.
\end{abstract}
\vspace{3em}
\begin{center}
\includegraphics[width=0.6\textwidth]{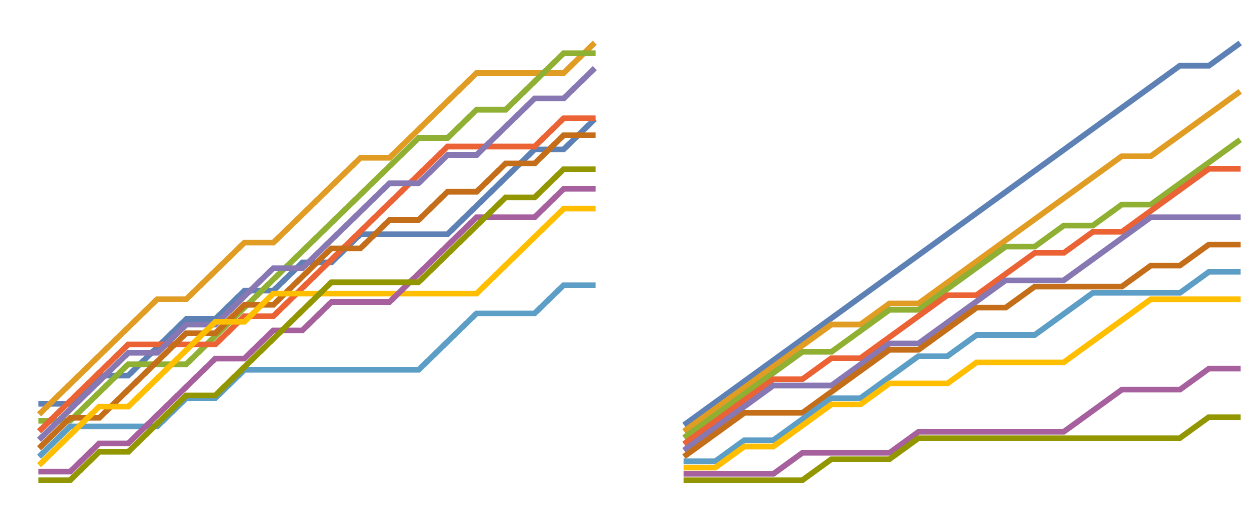}
\end{center}

	\tableofcontents	
	
\section{Introduction}


The Robinson-Schensted-Knuth (RSK) correspondence originates in the study of representations of the symmetric group,   \cite*{robinson1938representations}. In probability, it has been mainly used for understanding different models of two-dimensional directed last passage percolation. 

Across these models, several versions of RSK have been used in the literature. RSK has also been used to construct the directed landscape, an object that is expected to be the universal limit of such models, see \cite*{DOV}. The goal of this paper is to introduce a version of RSK that unifies three commonly used versions: ordinary RSK, dual RSK, and continuous RSK. We use this unified setting to prove some basic important properties of RSK: bijectivity and isometry.
The results in this paper are central to establishing the scaling limit of the longest increasing subsequence and other KPZ models in \cite*{dauvergne2021scaling}.  Our approach keeps probability applications in mind and avoids representation theory concepts entirely. We work with a global perspective, using last passage values as opposed to local bumping algorithms. 

One important new contribution of this work is in understanding an infinite-time version of the RSK bijection. Surprisingly, this version turns out to be much simpler than ordinary RSK! We also obtain parallel descriptions for RSK and its inverse, allowing us to give a purely global geometric description for the RSK inverse map. 

Let $\mathcal D^n$ be the space of $n$-tuples of cadlag functions from $[0,\infty)\to \R$ with no negative jumps.  Each $f\in \mathcal D^n$ defines a  finitely additive signed measure $df$ on $[0, \infty)\times \{1, \dots, n\}$  through
$$
df\left([x,y]\times\{i\}\right)=f_{i}(y) - f_i(x^-).
$$ 
When visualizing $f$ and this measure, we will think in matrix coordinates, so that line $1$ is on top and line $n$ is on the bottom.
For $(p, q) = (x, n; y, m) \in (\R \X \Z)^2$, we write $p \nearrow q$ if $x \le y$ and $n \ge m$.  For $p \nearrow q$, a path $\pi$ from $p$ to $q$ is a union of closed intervals 
\begin{equation*}
[t_i,t_{i-1}]\times \{i\}, \quad i=m,m+1,\ldots,n, \quad x=t_n\le t_{n-1}\le \cdots \le t_m \le  t_{m-1}=y.
\end{equation*}
Two paths are called {\bf essentially disjoint} if the corresponding intervals have disjoint interiors, see Figure \ref{fig:essentially}. Define the {\bf length} of a path $\pi$ by
$$
|\pi|_f =  df(\pi).
$$
For $u=(p;q)=(x,n;y,m)\in (\R \X \Z)^2$  define the \textbf{distance} in $f$ from $p$ to $q$ by 
\begin{equation}
\label{E:lpp-def-intro}
f[u]=f[p\to q]=f[(x, n) \to (y, m)] = \sup_{\pi} |\pi|_f,
\end{equation}
where the supremum is taken over all paths $\pi$ from $p$ to $q$. We also define $f[p^k\to q^k]=\sup df(\pi_1\cup \ldots \cup \pi_k)$, where the supremum is over tuples of $k$ essentially disjoint paths from $p$ to $q$. We call a tuple that achieves $f[p^k\to q^k]$ a \textbf{$k$-disjoint optimizer} from $p$ to $q$.
\begin{figure}[t]
\centering
\includegraphics[width=0.7\textwidth]{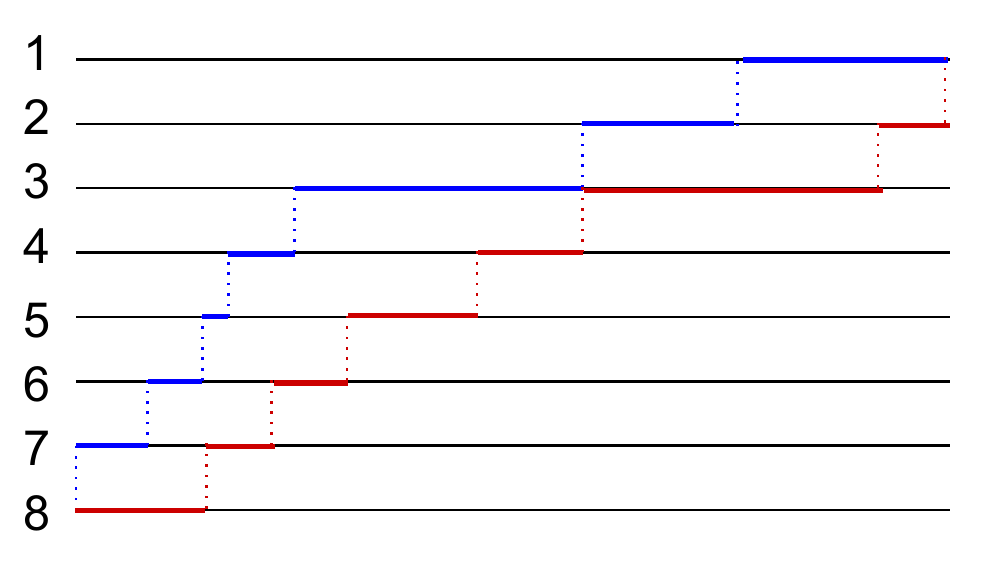}
\caption{Two essentially disjoint paths}
\label{fig:essentially}
\end{figure}

Define the \textbf{melon map} $W:\sD^n \to \sD^n$ by
\begin{equation}
\label{E:melon-def}
Wf_k(y) = f[(0, n)^k \to (y, 1)^k]- f[(0, n)^{k-1} \to (y, 1)^{k-1}]
\end{equation}
with the convention that  $f[p^0\to q^0]=0$, see the end of the introduction for discussion on how this is related to the standard presentation of RSK. 

We summarize some of the remarkable properties of the map $W$ in the next theorem. 
\begin{theorem}
	\label{T:W-facts}
	Consider the melon map $W$ on $\sD^n$. \begin{enumerate}[label=(\roman*)]
		\item {\bf Isometry.} For $p = (x, n)$ and $q = (y, 1)$, we have 
		$
		f[p \to q] = Wf[p \to q].
		$
		\item {\bf Idempotent property.} $WW=W$.
		\item {\bf Image.} $\operatorname{Im }W=\sD^n_{\uparrow}$, the set of all $f\in \sD^n$ such that $f_{i-1}(y^-) \ge f_i(y)$ for all $i, y$.
		\item {\bf Bijection.} $W$ is a bijection between $\sD^n_{\downarrow}$ and  $\sD^n_{\uparrow}$, see below.
		\end{enumerate}
\end{theorem}

The ordering in Theorem \ref{T:W-facts} (iii)  gives the sequence $Wf_1, \dots,Wf_n$ the appearance of stripes on a watermelon. For this reason, we call $Wf$ the \textbf{melon} of $f$.
The isometry property extends to multi-point last passage values as in \eqref{E:melon-def}, see Proposition \ref{P:W=w}.

The set of functions $\sD^n_{\downarrow}\subset \sD^n$ on which $W$ is a bijection can be explicitly described. 
Define $\sD^n_{\downarrow -}$ as the set of functions on which the following holds: the $k$ \emph{lowest} constant paths $[0,\infty)\times\{i\}, i=k-n+1,\ldots,n$ are a local limit as $t\to \infty$ of a sequence of $k$-disjoint optimizers from $(0,n)$ to $(t,1)$, see Section \ref{S:bijectivity}. The set $\sD^n_{\downarrow}$ is the closure of $\sD^n_{\downarrow -}$ in the uniform topology.


For the inverse of $W$, let $R_t$ transform  the signed measure $df$ up to time $t$ by rotating the base space $[0, t]\X \{1, \dots, n\}$ by 180 degrees. We write $R_tf=g$ if $R_t(df)=R_t(dg)$, and let 
\begin{equation}\label{E:M-def-intro}
Mf=\lim_{t\to\infty} R_tWR_tf.
\end{equation}
\begin{theorem}[Explicit inverse]
	\label{T:M-facts}	$M$ is well-defined on $\sD^n$. Moreover we have
\begin{enumerate}[label=(\roman*)]
		\item {\bf Isometry.} For $p = (x, n)$ and $q = (y, 1)$, we have 
		$
		f[p \to q] = Mf[p \to q].
		$
		\item {\bf Idempotent property.} $MM=M$.
		\item {\bf Image.} $\operatorname{Im }M=\sD^n_{\downarrow}$.
		\item {\bf Bijection.} $M$ is a bijection between $\sD^n_{\uparrow}$ and  $\sD^n_{\downarrow}$ with inverse $W$.
\end{enumerate}
\end{theorem}


The following proposition provides a simple sufficient condition for $f$ to be in $\sD^n_\downarrow$. It implies that classical examples, such as i.i.d.\ random walks or Brownian motion paths are in
$\sD^n_{\downarrow}$, and so there is no information is lost by applying the map $W$.

\begin{prop}
\label{P:sourness-intro}
If $f\in \sD^n$ and for all $j$, the function $(f_{j+1}-f_j)^+$ is unbounded, then $f \in \sD^n_\downarrow$.
\end{prop}
\begin{remark}
One useful perspective on our description of RSK is to focus purely on the isometry. Put an equivalence relation on $\sD^n$ by letting $f \sim g$ if $f[(x, n) \to (y, 1)] = g[(x, n) \to (y, 1)]$ for all $x, y$. From this point of view, $W$ maps $f$ to the element of its equivalence class with the leftmost disjoint optimizers,
and $M$ maps $f$ to the element of its equivalence class with the rightmost disjoint optimizers, see Section \ref{S:max-element} for a more precise setup. When thinking in these terms, idempotence and bijectivity fall out naturally.
\end{remark}

We can embed a finite time RSK correspondence into the melon map $W$, see Section \ref{SS:finite-bijection} for details. Bijectivity and other properties in the finite setting can be deduced from the simpler infinite case. Restrictions of this finite time bijection recover the usual RSK and dual RSK correspondences, see Section \ref{S:embedding}. 



Bijectivity is the reason that certain measures on $\sD^n$ have tractable pushforwards under $W$, and this is why RSK is useful in  probability. For example, if $B \in \sD^n$ consists of $n$ independent standard Brownian motions, then $WB$ is simply $n$ independent standard Brownian motions conditioned so that $B_1(t) \ge \dots \ge B_n(t)$ at all times $t$.

These results are traditionally proven using determinants and Doob transforms. We give a new computation-free proof that relies only on the bijectivity of RSK for the case of Bernoulli walks in Theorems \ref{T:Bernoulli-map} and \ref{T:Gibbs-LLN}. The advantage of this approach is that the same argument works for all the different integrable models of RSK, and so it avoids one-off computations that are specific to the details, such as  discrete vs continuous time, of  individual models. In the following theorem, we use the piecewise linear embedding of Bernoulli walks in the space of continuous functions, see Figure \ref{fig:rsksmall}. 
\begin{figure}[t]
\centering
\includegraphics[width=0.7\textwidth]{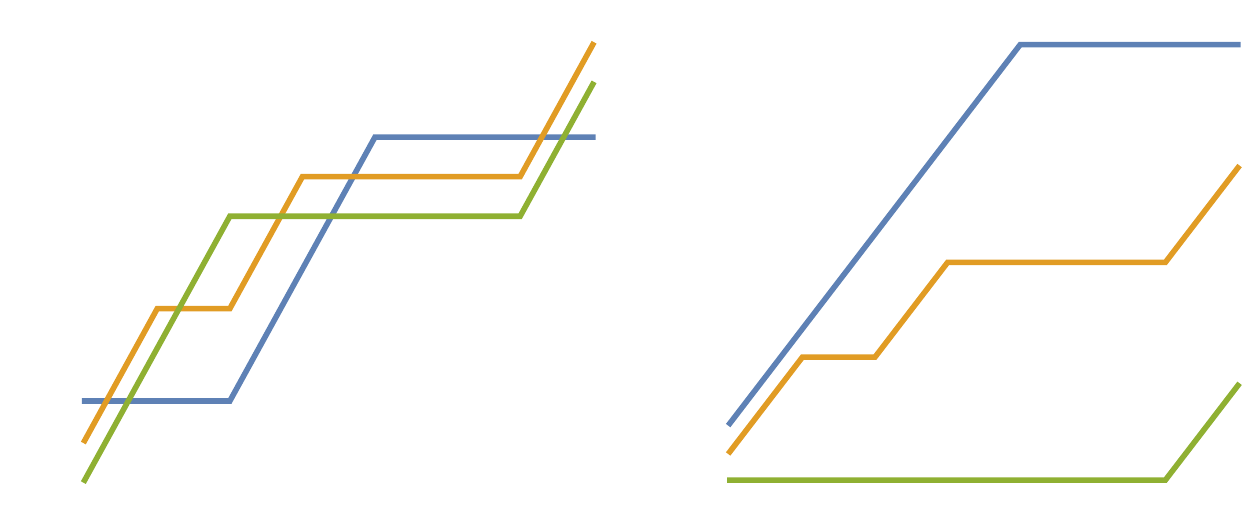}
\caption{Bernoulli walks $f$ and $Wf$}
\label{fig:rsksmall}
\end{figure}
\begin{theorem}
\label{T:Bernoulli-map-intro}
Let $Y \in \sD^n$ consist of independent Bernoulli random walks of drift $d\in \mathbb [0,1]^n$. Then the law of $WY \in \sD^n$ is nonintersecting Bernoulli walks with drift given by the order statistics of $d$. 
\end{theorem}

Theorem \ref{T:Bernoulli-map-intro} also follows from results of \cite{o2003conditioned}, which relate RSK to a Doob transform describing nonintersecting walks, see \cite{konig2002non} for the  Doob transform approach to nonintersecting walks. We believe the non-computational proof we present is new. 

There are several subsets $A\subset \sD^n$ of paths so that RSK is a bijection between $\sD^n_\downarrow \cap A$ and $\sD^n_\uparrow\cap A$. The following table informally summarizes these sets and natural measures on them. The measures then correspond to classical integrable last passage percolation models, see Sections \ref{S:bijectivity} and \ref{S:preservation}. ``Unit jumps'' means piecewise constant functions that have jumps of size 1 only; Bernoulli paths are continuous and linear of slope $0$ or $1$ on $[j,j+1]$, $j\in \{0,1,2,\ldots\}$ and the S-J model is the Sepp\"al\"ainen-Johansson model, see \cite*{seppalainen1998exact}. In all of these cases, the image under $W$ of the natural measure on paths can be interpreted as the same measure conditioned to fall in $\sD^n_{\uparrow}$.
\[\begin{array}{l|l|l} A	&  \text{independent walk measure on $\sD^n$}  & \text{LPP model} \\ \hline
	\text{continuous functions}
	& \text{Brownian motions} &\text{Brownian LPP} \\
	\text{unit jumps} 
	& \text{Poisson counting processes}  &\text{Poisson lines LPP} \\
\mathbb N\text{ jumps  at }\mathbb N\text{ times}
	& \text{discrete-time geometric random walks}  &\text{geometric LPP} \\
\mathbb R^+\text{ jumps  at }\mathbb N\text{ times}
	& \text{discrete-time exponential random walks} &\text{exponential LPP} \\
	\text{Bernoulli paths}
	& \text{piecewise linear Bernoulli walks} &\text{S-J model}
	\end{array}
	\]
These examples also show how versions of classical RSK embed into the present framework of RSK. More precisely, usual RSK corresponds to piecewise constant nonnegative integer jumps at integer times, dual RSK corresponds to Bernoulli paths, and continuous RSK along the lines of \cite*{biane2005littelmann} embeds as continuous paths. See Section \ref{S:embedding} for proofs and more details.

\subsection*{Background}

A version of RSK first appeared in  \cite*{robinson1938representations}. The bijection was later extended in \cite*{schensted1961longest}, and in \cite*{knuth1970permutations}. Classical treatments of RSK can be found in \cite*{stanley2},  \cite*{fulton1997young}, \cite*{romik2015surprising} and \cite*{sagan2013symmetric}. \cite*{schensted1961longest} and \cite*{greene1974extension}  tied RSK to longest increasing subsequences, and therefore last passage percolation, see \cite*{vershik1977asymptotics} and \cite*{logan1977variational}. We use Greene's description as the {\it definition} of RSK. An independent line of research started with the discovery of Pitman's $2M-X$ theorem, \cite*{pitman1975one}. The two ideas were first unified in depth in \cite*{biane2005littelmann}. That work has versions of many of the results presented here, and is rooted in representation theory -- part of our goal is to give a treatment where concepts of  representation theory are not prerequisite.


Versions of the isometry property and its strengthening (Proposition \ref{P:W=w}.(i)) were shown in \cite*{noumi2002tropical}, and also in \cite*{biane2005littelmann}, and \cite*{DOV}. Theorem \ref{T:W-facts} (ii) is more classical, and can be shown with a path-crossing argument. 

There are other  generalizations for RSK. Geometric RSK is a finite temperature version of ordinary RSK initiated by \cite*{kirillov2001introduction}, see also \cite*{noumi2002tropical}, \cite*{corwin2014tropical}. \cite*{noumi2002tropical} have finite and zero temperature versions of many of the results presented here, obtained using matrix methods.

In particular, isometry in the geometric setting was shown in \cite*{noumi2002tropical}, see also \cite*{corwin2020invariance} and \cite*{dauvergne2020hidden}. Further extensions, including randomized versions are studied in  \cite*{o2013q}, \cite*{bufetov2018hall},  \cite*{garver2018minuscule}, 
\cite*{aigner2020q}, and \cite*{dauvergne2020hidden}.








\section{Percolation across cadlag functions}
	\label{S:lpp-cadlag}
	
%
\subsection{Basic definitions}

Recall that a function $f$ from an interval $I$ to $\R$ is {\bf cadlag} if for all $x \in I$, we have 
\begin{equation}
\label{E:fyfx}
\lim_{y \to x^{+}} f(y) = f(x), \quad \mathand \quad \lim_{y \to x^{-}} f(y) \text{ exists }.
\end{equation}
Note that either one of these limits may not be defined if $x$ is an endpoint of $I$. We write $f(x^-)$ for the second limit in \eqref{E:fyfx}. When $f(x^-) \ne f(x)$, we say that $f(x) - f(x^-)$ is a \textbf{jump} of $f$ and that $x$ is a \textbf{jump location}.
Cadlag functions can only have countably many jumps. Let $\sD^n$ be the space of all functions  
	$$
	f:[0, \infty) \X \{1, \dots, n\} \to \R, \qquad (x,i)\mapsto f_i(x).
	$$
	so that each $f_i$ is a cadlag function whose jumps are all positive and satisfies $f_i(0) \ge 0$. 
	We impose that $f_i(0^-) = 0$ for all $i$. If $f_i(0) > 0$, we interpret this as $f$ having a jump at $0$. The boundary condition $f_i(0^-) = 0$ is simply a convention for us since we will only care about the increments of $f$. We will often think of $f$ as a sequence of functions $f_1, \dots, f_n$. When visualizing $f$  we will think in matrix coordinates, so that line $1$ is on top and line $n$ is on the bottom.

	We associate to any $f \in \sD^n$ a finitely additive signed measure $df$ on $[0, \infty) \X \{1, \dots, n\}$ given by
	$$
	df\left([x,y]\times\{i\}\right)=f_{i}(y) - f_i(x^-).
	$$ 
Our boundary convention $f_i(0^-) = 0$ means that we can always reconstruct $f \in \sD^n$ from its measure $df$. 
	
	Now, for $p = (x, n), q = (y, m) \in [0, \infty) \X \Z$ with $x \le y, n \ge m$, a {\bf path} $\pi$ from $p$ to $q$ is a union of closed intervals 
	\begin{equation}
	\label{E:intervals}
	[t_i,t_{i-1}]\times \{i\}\subset   \mathbb R \times I, \qquad i=m,m+1,\ldots,n,
	\end{equation}
	where
	\begin{equation}
	\label{E:jumptimes}
	x=t_n\le t_{n-1}\le \cdots \le t_m \le  t_{m-1}=y,
	\end{equation}
see Figure \ref{fig:essentially}.
	The points $t_i, i = m, \dots, n - 1$ are called the {\bf jump times} of $\pi$. For $f \in \sD^n$ and a path $\pi$ contained in $\R\X \{1, \dots, n\}$, we can define the \textbf{length} of $\pi$ with respect to $f$ by 
	$$
	|\pi|_f = df(\pi) = \sum_{i=m}^{n} f_{i}(t_{i-1}) - f_i(t_{i}^-).
	$$
	This definition is chosen so that all the jumps of $f$ that lie along the path $\pi$ are accounted for. For $f \in \sD^n$ and $u=(p;q)=(x,n;y,m)\in (\R \X \{1, \dots, n\})^2$  define the \textbf{last passage value} across $f$ from $p$ to $q$ by
	\begin{equation}
	\label{E:lpp-def}
	f[u]=f[p\to q]=f[(x, n) \to (y, m)] = \sup_{\pi} |\pi|_f,
	\end{equation}
	where the supremum is taken over all paths $\pi$ from $p$ to $q$. If no path from $p$ to $q$ exists, we set $f[u] = -\infty$. We call a path $\pi$ from $p$ to $q$ a \textbf{geodesic} if $|\pi|_f = f[p \to q]$.

	\subsection{Multiple paths}
	
	Next, we generalize last passage values to multiple paths. 
	First, for two paths $\pi$ and $\rho$, We say that $\pi$ is {\bf to the left of} $\rho$ if for every $(x,\ell)\in \pi$ and $(y,m)\in \rho$ at least one of the inequalities $\ell \le m$ and $x\le y$ holds. Equivalently, we say that $\rho$ {\bf is to the right of} $\pi$. 
	We say $\pi, \rho$  are {\bf essentially disjoint} if the set $\pi \cap \rho$ is finite. Recall that
	we think in matrix coordinates, so that line $1$ is on top and line $n$ is on the bottom, see Figure \ref{fig:essentially}.

	Let $\mathbf{p} = (p_1, \dots, p_k), \mathbf{q} = (q_1, \dots, q_k) \in ([0, \infty) \X \Z)^k$. 
	A \textbf{disjoint $k$-tuple (of paths)} $\pi=(\pi_1,\ldots ,\pi_k)$ from $\mathbf{p}$ to $\mathbf{q}$ is defined by the following properties:
	\begin{itemize}[nosep]
		\item $\pi_i$ is a path from $p_i$ to $q_i$,
		\item $\pi_i$ is to the left of $\pi_j$ for $i < j$,
		\item $\pi_i$ and $\pi_j$ are essentially disjoint for all $i \ne j$.
	\end{itemize}
	For $f \in \sD^n$ and a disjoint $k$-tuple $\pi$ let $\cup \pi := \pi_1 \cup \dots \cup \pi_k \sset [0, \infty) \X \{1, \dots, n\}$, and define the \textbf{length} of $\pi$ by
	$$
	|\pi|_f=df(\cup \pi).
	$$
	For $\mathbf u = (\mathbf p, \mathbf q)$, we can then define the {\bf multi-point last passage value}
	$$
	f[\mathbf u] = f[{\bf p \to \bf q}] = \sup_{\pi} |\pi|_f,
	$$
	where the supremum is over disjoint $k$-tuples from $\bf p$ to $\bf q$. Again, if no such $k$-tuples exist, we set $f[\mathbf u] = -\infty$.
	We call a $k$-tuple $\pi$ satisfying $f[{\bf p \to \bf q}] = |\pi|_f$ a {\bf disjoint optimizer}. 
	
	\subsection{Basic geometric properties}
	
Next, we collect some basic geometric facts about last passage paths. We start by showing that disjoint optimizers always exist. As in the previous section, $f \in \sD^n$ and let $\mathbf{p} = (p_1, \dots, p_k), \mathbf{q} = (q_1, \dots, q_k) \in ([0, \infty) \X \Z)^k$.

	\begin{lemma}
		\label{L:multipath-exist} If there is at least one disjoint $k$-tuple from $\bp$ to $\bq$, then there exists a disjoint optimizer for $f$ from $\mathbf p$ to $\mathbf q$.
	\end{lemma}

Lemma \ref{L:multipath-exist} is an immediate consequence of the following two  observations. Each of these next lemmas is also useful in its own right.

\begin{lemma} [Compactness]
	\label{L:cty-lpp}
	The space of disjoint $k$-tuples from $\mathbf p$ to $\mathbf q$ is compact in the Hausdorff topology on $([0, \infty) \X \{1, \dots, n\})^k$.
\end{lemma}

Lemma \ref{L:cty-lpp} is immediate from the definitions. 

\begin{lemma} [Upper semicontinuity]
\label{L:usc-lpp}
For any $f \in \sD^n$, the function $\pi \mapsto |\pi|_f$ mapping disjoint $k$-tuples to their $f$-length is upper semicontinuous in the Hausdorff topology.
\end{lemma}

Lemma \ref{L:usc-lpp} is a consequence of the fact that if $a_m \to a, b_m \to b$, then since $f \in \sD^n$ has only positive jumps, 
$$
\limsup_{m \to \infty} df([a_m, b_m] \X \{i\}) \le df([a, b] \X \{i\}). 
$$
Next, we give a metric composition law. The proof is again immediate from the definitions. For this lemma and in the sequel, we use the shorthand notation  $$(t, \mathbf m) = \Big((t, m_1),   \dots, (t, m_k)\Big) \quad \text{and similarly} \quad (\bx, \ell) = \Big((x_1, \ell), \dots,(x_k, \ell)\Big).$$

\begin{lemma}
	[Metric composition law]\label{L:split-path}
	Let $f \in \sD^n$, let $(\bp, \bq) = (\bx, \ell; \by, m) \in ([0, \infty) \X \{1, \dots, n\})^k$ and let $i \in \{m + 1, \dots, \ell\}$. Then
	$$
	f[\bp \to \bq] = \max_{\bz \in [0, \infty)^k} f[\bp \to (\bz, i)] +  f[(\bz, i - 1) \to \bq].
	$$
\end{lemma}
More general versions of the metric composition law exist, though we do not require them here. Lemma \ref{L:split-path} implies certain triangle inequalities for last passage values, reinforcing the idea that the last passage structure is best thought of as a metric.

We end this section with an extremely useful \textbf{quadrangle inequality} for multi-point last passage values. This inequality generalizes a well-known quadrangle inequality for single-point last passage values.

\begin{lemma}
	\label{L:quadrangle}
	Let $(\bp, \bq) = (\bx, n; \by, m), (\bp', \bq') = (\bx', n; \by', m) \in ([0, \infty) \X \{1, \dots, n\})^{2k}$ be such that $x_i \le x_i', y_i \le y_i'$ for all $i$. Then
	$$
	f[\bp \to \bq'] + f[\bp' \to \bq] \le f[\bp \to \bq] + f[\bp' \to \bq'].
	$$
\end{lemma}

This is a special case of Lemma 2.4 in \cite*{dauvergne2021disjoint} generalized to the cadlag setting.  
\begin{proof}
	Let $\pi$ be a disjoint optimizer from $\bp$ to $\bq'$, and let $\pi'$ be a disjoint optimizer from $\bp'$ to $\bq$. We can define disjoint $k$-tuples $\tau^\ell, \tau^r$ as follows. For each $i$, let $\tau^\ell_i$ be the leftmost path from $p_i$ to $q_i$ contained in the union $\pi_i \cup \pi_i'$ and let $\tau^r_i$ be the rightmost path from $p_i'$ to $q_i'$ contained in $\pi_i \cup \pi_i'$. We can think of $\tau^\ell_i, \tau^r_i$ as order statistics of $\pi_i, \pi'_i$. With this construction, $\tau^\ell$ is a disjoint $k$-tuple from $\bp$ to $\bq$, $\tau^r$ is a disjoint $k$-tuple from $\bp'$ to $\bq'$, and 
	\begin{equation*}
	|\pi|_f + |\pi'|_f = |\tau^\ell|_f + |\tau^r|_f.
	\end{equation*}
	The left side above equals $f[\bp \to \bq'] + f[\bp' \to \bq]$ and the right side is bounded above by $f[\bp \to \bq] + f[\bp' \to \bq']$.
\end{proof}

A similar proof idea to Lemma \ref{L:quadrangle} shows that rightmost and leftmost optimizers always exist. Again, this is a generalization of Lemma 2.2 in \cite{dauvergne2021disjoint} to the cadlag setting. To state the lemma, for disjoint $k$-tuples $\lambda, \pi$, we write $\lambda \le \pi$ and say that $\la$ is \textbf{to the left} of $\pi$ if $\lambda_i$ is to the left of $\pi_i$ for every $i$.

\begin{lemma}
	\label{L:leftmost-exist}
		Let $(\bp, \bq) = (\bx, n; \by, m)$ be such that there is at least one disjoint $k$-tuple from $\bp$ to $\bq$. Then for any $f \in \sD^n$, there are optimizers $\rho, \lambda$ across $f$ from $\bp$ to $\bq$ such that for any optimizer $\pi$ from $\bp$ to $\bq$, we have $\rho \le \pi \le \lambda$. We call $\rho, \lambda$ the \textbf{rightmost} and \textbf{leftmost} optimizers from $\bp$ to $\bq$.
\end{lemma}

\begin{proof}
	Consider the set $S$ of all optimizers from $\bp$ to $\bq$ with the partial order $\le$. Lemmas \ref{L:cty-lpp} and \ref{L:usc-lpp} imply that any totally ordered subset of $S$ has upper and lower bounds. Therefore by Zorn's lemma, $S$ contains at least one minimal element. Suppose that $\pi, \pi'$ are both minimal elements. Construct disjoint $k$-tuples $\tau^\ell, \tau^r$ from $\bp$ to $\bq$ as in the proof of Lemma \ref{L:quadrangle} so that $\tau^\ell \le \pi, \pi'$ and
	\begin{equation}
	\label{E:fpq}
	2 f [\bp \to \bq] = |\pi|_f + |\pi'|_f = |\tau^\ell|_f + |\tau^r|_f.
	\end{equation}
	Since $|\tau^\ell|_f, |\tau^r|_f  \le f [\bp \to \bq]$ by definition of $f [\bp \to \bq]$, \eqref{E:fpq} can only hold if both $\tau^r, \tau^\ell$ are optimizers from $\bp$ to $\bq$. Therefore $\tau^\ell \in S$ so by the minimality of $\pi, \pi'$ we have $\pi = \tau^\ell = \pi'$. Therefore $S$ contains a unique minimal element, the leftmost optimizer. The existence of the rightmost optimizer follows by a symmetric argument. 
\end{proof}
\section{The melon}

\label{S:the-melon}
Recall that the melon $Wf \in \sD^n$ of a function $f \in \sD^n$ is given by
\begin{equation}
\label{E:melon-def-body}
Wf_k(y) = f[(0, n)^k \to (y, 1)^k]- f[(0, n)^{k-1} \to (y, 1)^{k-1}]
\end{equation}
with the convention that  $f[p^0\to q^0]=0$.
We say that $f,g\in \sD^n$ are {\bf isometric} if
$$
f[(\bx, n) \to (\by, 1)] = g[(\bx, n) \to (\by, 1)]
$$
for all $k$-tuples $\bx, \by$. In other words, last passage values between the top and bottom boundaries in the environments defined by $f$ and $g$ are equal. Isometry is an equivalence relation on $\sD^n$, which we denote by $f \sim g$. 

The main goal of the section is to show that $f$ is isometric to $Wf$. To do this, we will show that $Wf$ agrees with iterated applications of $2$-line melon maps to $f$, alternately known as \textbf{Pitman transforms}. 
 
%
%
%
%

	\subsection{The Pitman transform}
	\label{S:Pitman}
From the definition of the melon map for $n=2$ we immediately get the following result.
	\begin{lemma}[The Pitman transform]
	\label{L:Melon-2}
	For $f \in \sD^2$, we have
		\begin{eqnarray}
		Wf_1(x) &=& f[(0,2) \to (x,1)]\\
		Wf_2(x) &=& f_1(x)+f_2(x) - Wf_1(x). 
		\end{eqnarray}
	\end{lemma}
	
Our main goal in Section \ref{S:Pitman} is to prove that in the $2$-line case, $Wf \sim f$. The first step is
	Lemma \ref{L:wm-lem}, which shows that $W$ preserves single-point last passage values. This has the most technical proof in the paper, and consists of careful manipulations of the definitions.  The proof is complicated by the fact that the function $f$ is only cadlag, rather than continuous, see Lemma 4.2 in \cite*{DOV} for the easier continuous case. We defer to the proof to the Appendix \ref{S:appendix}.
		\begin{lemma}
		\label{L:wm-lem}
		The map $W$ maps $\sD^2$ to itself. Moreover, for $f \in \sD^2$ and $0 \le x \le y$, we have that
		\begin{equation}
		    	f\big[(x, 2) \to (y, 1) \big] = Wf\big[(x, 2) \to (y, 1) \big]
		\end{equation}
	\end{lemma}

	Lemma \ref{L:wm-lem} and the definition of $W$ allow us to identify both the image of the two-line map $W$, and the fact that it is idempotent.
	
	\begin{definition}
	We say that two functions $f_1, f_2\in \sD$ are \textbf{Pitman ordered},  if 
	$$
	f_2(t)\le f_1(t^-) \quad \text {for all }t.
	$$
	When this is the case, we will write $f_2 \preceq f_1$.
	\end{definition}
	
	\begin{remark}
	\label{R:Pitman-ordered}
	    With this definition the set $\sD_\uparrow^n$ introduced in Theorem \ref{T:W-facts} (iii) can be written simply as $\sD_\uparrow^n = \{f \in \sD^n : f_n \preceq \ldots \preceq f_1 \} \subset \sD^n$. Since the lines are ordered in this way, it is clear that for $f \in \sD^n_\uparrow$ there is a disjoint optimizer from $(0, n)^k \to (y, 1)^k$ following the leftmost possible paths, i.e.$
	    f_1(x)+\ldots+f_k(x) = f[(0,n)^k\to(x,1)^k]$. In short, 
	    \begin{equation}
	    \label{E:DuW}
	        Wf=f
	    \end{equation}
	    for $f \in \sD^n_\uparrow$.
	\end{remark}
	
	\begin{lemma}
		\label{L:idempotency} Let $f=(f_1,f_2) \in \sD^2$. The following are equivalent:
		\begin{enumerate}[label=(\roman*), nosep]
			\item $(Wf)_1=f_1$,
			\item $(Wf)_2=f_2$,
			\item $Wf=f$,
			\item $f_2\preceq f_1$.
		\end{enumerate}
		Moreover, the Pitman transform is idempotent, $W^2 = W$ and the the image of Pitman transform is  precisely the set of all Pitman ordered functions: $$W(\sD^2) = \sD^2_\uparrow = \{f \in \sD^2 : f_2 \preceq f_1 \}.$$
	\end{lemma}
	
	\begin{proof}
	By Lemma \ref{L:Melon-2}, $Wf_1 + Wf_2 = f_1 + f_2$. This gives the equivalence of (i), (ii) and (iii). For (iv), observe that $f_2\preceq f_1$ if and only if 
	$$
	f[(0, 2) \to (x, 1)] = f_1(x) +	\sup_{0 \le y \le x} f_2(y) - f_1(y^-) = f_1(x)$$ for all $x$, or in other words $Wf_1=f_1$, giving the equivalnce of (i) and (iv). 
	Lemma \ref{L:wm-lem} then implies the claims about idempotency and image.
	\end{proof}

	Next, we extend Lemma \ref{L:wm-lem} to deal with disjoint $k$-tuples. For the proof, we need the notion of a path from $p^+$ to $q$. This is defined the same way as the path $\pi$ from $p=(x,n)$ to $q=(y,m)$ as in \eqref{E:intervals}, with the vertical line 
	$\{x\}\times\{m,\ldots,n\}$ removed. The last passage value
	$
	f[p^+ \to q] 
	$
	is the supremum of $|\pi|_f$ over all paths $\pi$ from $p^+$ to $q$.
	Analogously, we define paths from $p$ to $q^-$ and $p^+$ to $q^-$. paths, and last passage values $f[p\to q^-]$ and $f[p^+ \to q^-]$. We can think of the last passage value $f[p \to q]$ as $f[p^- \to q^+]$.
	
	\begin{corollary}\label{C:endpoints} Let $p=(x,2)$, $q=(y,1)$ with $0\le x<y$. Then 
		$$
		Wf[p^+\to q]=f[p^+\to q], \quad Wf[p\to q^-]=f[p\to q^-], \quad Wf[p^+\to q^-]=f[p^+\to q^-].
		$$
	\end{corollary}
	\begin{proof}
		Since the functions $f_i$ and $Wf_i$ are cadlag, we have
		$$f[p^+ \to q]=\lim_{z\downarrow x}f[(z,2)\to q], \qquad \mathand \qquad Wf[p^+ \to q]=\lim_{z\downarrow x}Wf[(z,2)\to q].
		$$
		Lemma \ref{L:wm-lem} then implies the first claim. The others are proven analogously. 
	\end{proof}
	
	\begin{lemma}[Isometry of the Pitman transform]
		\label{L:n2-general}
		For $f \in \sD^2,$ $Wf \sim f$. That is, for every pair of $k$-tuples $(\mathbf{p}, \mathbf{q}) =(\bx, 2; \by, 1)$, we have 
		$$
		f\big[\mathbf{p} \to \mathbf{q} \big] = Wf[\mathbf{p} \to \mathbf{q}].
		$$
	\end{lemma}
	
	\begin{proof}
		Define  $N:\R \to \{0, 1, \dots\}$ by $N(z) = \#\{i: z \in [x_i, y_i]\}$.  Disjointness of a $k$-tuple $\pi$ from $\mathbf{p}$ to $\mathbf{q}$ implies the following. 
		\begin{itemize}[nosep]
			\item Let $Z=N^{-1}( \{2, \dots\})$, that is the set of points $z$ that are contained in at least two intervals $[x_i,y_i]$. Then  $Z\times\{1,2\}\subset \cup \pi$.
			\item Let $I_1, \dots, I_\ell$ be the connected components of $N^{-1}(1)$, let $\hat a_j \le \hat b_j$ be the endpoints of $I_j$ and set $a_j = (\hat a_j, 2), b_j = (\hat b_j, 1)$. Then $I_j \X \{1, 2\} \cap (\cup \pi)$ is a path from $a_j^*$ to $b_j^*$ for all $j$. Here $* \in \{+, -\}$, depending on whether the corresponding end of $I_j$ is open or closed.
			\item $(\cup \pi)\cap (N^{-1}(\{0\}) \X \{1, 2\}) = \emptyset.$
		\end{itemize}
		Therefore letting $\rho_i$ be the path $I_j \X \{1, 2\} \cap (\cup \pi)$ from $a_j^*$ to $b_j^*$, we can write
		\begin{equation}
		\label{E:union-decomp}
		\cup \pi = Z\times\{1,2\} \,\cup\, \bigcup_{i=1}^\ell \rho_i.
		\end{equation}
		Moreover, given arbitary $a_i^* \to b_i^*$ paths $\rho_i$, there is a $\mathbf{p}\to\mathbf{q}$ path $\pi$ so that \eqref{E:union-decomp} holds. Therefore
		$$
		f[\mathbf{p} \to \mathbf{q}] = df(Z\times\{1,2\}) + \sum_{i=1}^\ell f[a_i^* \to b_i^*].
		$$
		By the preservation of the sum $f_1 + f_2$, the first term doesn't change when we apply $W$. By Corollary \ref{C:endpoints}, nor does the remaining sum.
	\end{proof}

	\subsection{Repeated Pitman transforms, cars, and the melon}
	\label{SS:repeated-pitman}
	Next, we build the $n$-line melon map. We first extend the Pitman transform to functions $f \in \sD^{n}$ by applying it to two lines at a time. For $i \in \{1, \dots, n-1\}$, define $\sig_i:\sD^n \to \sD^n$ by
	$$
	\sig_i f = (f_1, \dots, f_{i-1}, W(f_i, f_{i, i+1})_1, W(f_i, f_{i, i+1})_2, f_{i+2} \dots, f_n).
	$$
	
	\begin{prop}[Isometry property of $\sigma_i$]
		\label{P:wm-equivalent}
		Let $f \in \sD^n$, $n, k \in \N$, and let $(\bp, \bq) = (\bx, n; \by, 1)$. Then for $m \in \{1, \dots, n-1\}$, we have
		\begin{equation}
		\label{E:wm-equiv-lpp}
		f[\mathbf{p} \to \mathbf{q}] = \sig_m f[\mathbf{p} \to \mathbf{q}].
		\end{equation}
	\end{prop}
	
	\begin{proof}
		We first assume $m \in \{2, \dots, n-2\}$.
		By the metric composition law, Lemma \ref{L:split-path}, applied twice, we can write
		\begin{equation}
		\label{E:split-at-m}
		f[\mathbf{p} \to \mathbf{q}] =  \sup_{\bz, \bw \in \R^k} \Big(f[\mathbf{p} \to (\bz, m+2)] + f[(\bz, m+1) \to (\bw, m)] +  f[(\bw, m-1) \to \mathbf{q}] \Big).
		\end{equation}
		For a fixed pair $(\bz, \bw)$, when we apply $\sig_m$ to $f$, the first and third terms under the supremum in \eqref{E:split-at-m} do not change since the relevant components $f_i$ do not change. The middle term is preserved under the transformation $\sig_m$ by Lemma \ref{L:n2-general}. Hence the right hand side of \eqref{E:split-at-m} is also preserved under the map $f \mapsto \sig_m f$. The cases $m =1, n-1$ are similar with one of the terms in \eqref{E:split-at-m} removed.
	\end{proof}

For $m \le n$, define $\tau_m:\sD^n \to \sD^n$ by
$$
\tau_m=\sigma_m\sigma_{m+1}\cdots \sigma_{n-1}.
$$
We will build the $n$-line Pitman transform from composing the maps $\tau_i$. 

	
\begin{remark}[Cars]
	\label{R:cars}
The maps $\tau_m$   can be used to give a connection between particle systems and last passage percolation as follows. The functions $t \mapsto \tau_kf_k(t)$, with $k=1,\ldots, n$ can be thought of as  deterministic versions of the totally asymmetric simple exclusion process, tasep. 
		
Informally, think of $n$ cars moving on a single-lane highway in the same, typically negative direction. The cars cannot pass each other. The derivative $f'_k(t)$ is the desired velocity of car $k$ at time $t$.  
		
However, cars cannot always move at their desired velocity. Cars ignore cars behind them, but are often forced to avoid cars ahead of them, in which case they slow down just enough to avoid collision. The key feature of these models is that the interaction, rather than the direction of movement is totally asymmetric. Slowing down could mean being  forced to back up if the car ahead does. The function $\tau_kf_k(t)$ is the position of the $k$th car at time $t$.
		
A different model for the movement of the cars is that ignore the cars ahead of them and are forced to speed up to avoid collisions from cars behind them. For this model, the location of  car $k$ is then given by  $\sigma_{k-1}\cdots \sigma_{1}f_k(t)$. 
This version is push-asep and first passage percolation. 
		
Stochastic models that fit in this setting include tasep, discrete-time tasep,  push-asep (which also has totally asymmetric interactions), Brownian tasep and the Hammersley process. 
	\end{remark}
	
	In light of the remark, the following lemma states a deterministic equivalence between exclusion processes and last passage percolation. 
	\begin{lemma}
		\label{L:top-LP-preserve}
		Let $f \in \sD^n$ and $1 \le m \le n \in \N$. Then for $y \ge 0$, we have
		\begin{equation}\label{E:TASEP}
		f[(0,n)\to (y,m)]=\tau_mf_m(y).    \end{equation}
	\end{lemma}
	
	\begin{proof}
		We show this by induction on $n-m$. The $n=m$ case is true by definition of $\tau_n = Id$. For $m < n$, we have by the metric composition law, Lemma \ref{L:split-path}, 
		\begin{align*}
		f[(0,n)\to(y,m)]&=\sup_{z\le y} f[(0,n)\to(z,m+1)]-f_m(z^-)+f_m(y)
		\\
		&=\sup_{z\le y}\tau_{m+1} f_{m+1}(z)  -f_m(z^-)+f_m(y)\\
		&= \sig_m \tau_{m+1} f_m(y),
		\end{align*}
		where the second equality follows from the inductive hypothesis, and the third equality is simply the definition of $\sig_m$.
		By the definition of $\tau_m$ this equals the right  hand side of \eqref{E:TASEP}.
	\end{proof}
	
	\begin{definition}
	\label{D:iterated-melon-def}
	We will build the $n$-line map $\omega:\sD^n \to \sD^n$ from the functions $\tau_m$ as follows. Define
	\begin{equation}\label{eq:w}
	\omega=\tau_{n-1}\tau_{n-2}\cdots \tau_1 = \sig_{n-1} \sig_{n-2} \sig_{n-1} \cdots \sig_1 \sig_2 \cdots \sig_{n-1}.
	\end{equation}
	Informally speaking, if we think of $\sig_i$ as ``sorting'' the $i$-th and $i+1$-st functions $f_i$ and $f_{i+1}$, then $\omega$ is precisely the ``bubble sort'' algorithm to sort the entire ensemble $f$.
	\end{definition}
	
	In the following proposition, we show that $\omega=W$, the melon map from \eqref{E:melon-def-body}.
	
	\begin{prop}
		\label{P:W=w}
		Let $\omega f$ be the melon of a function $f \in \sD^n$. Then:
		\begin{enumerate}[label=(\roman*)]
			\item (Isometric equivalence) $\omega f \sim f$.
			\item (Idempotence) $\omega^2 = \omega$.
			\item (Image) We have $
			\omega (\sD^n) = \sD^n_\uparrow = \{f \in \sD^n : f_n \preceq \dots \preceq f_1\}.$
			\item $\omega =W$ from  \eqref{E:melon-def-body}.
		\end{enumerate}
	\end{prop} 
	
	To prove Proposition \ref{P:W=w} we need two short lemmas.
	
	\begin{lemma}
		\label{L:tm2}
		Let $m < n \in \N$. Then  $\tau_m^2=\tau_{m+1}\tau_m$.
	\end{lemma}
	\begin{proof}
		Let $g=\tau_{m+1}\tau_mf$. We want to show that $g=\sigma_mg$. The functions $g$ and $\sigma_mg$ can only differ in the coordinates $m,m+1$. By Lemma \ref{L:idempotency} it suffices to show that $g_m=\sigma_mg_m$. 
		
		By Lemma \ref{L:top-LP-preserve}, we have 
		$$
		\sigma_mg_m(y)=\tau_m\tau_mf_m(y)=\tau_mf[(0,n)\to (y,m)]=f[(0,n)\to (y,m)],
		$$
		where the last equality is by Proposition \ref{P:wm-equivalent} repeatedly applied to $(f_m,\ldots, f_n)$.
		Since $\tau_{m+1}$ does not change coordinate $m$, Lemma \ref{L:top-LP-preserve} gives 
		\[
		g_m(y) = \tau_{m+1}\tau_mf_m(y)=\tau_mf_m(y)=f[(0,n)\to (y,m)]. \qedhere
		\] 
	\end{proof}

	\begin{lemma}\label{L:w-longest}
		For $m=1,\ldots, n-1$ we have
		$\sigma_m \omega=\omega$.
	\end{lemma}
	\begin{proof}
		The statement for $m = n-1$ follows since $\sig_{n-1}^2 = \sig_{n-1}$ by Lemma \ref{L:idempotency}. Now assume $m < n-1$.
		When $i-j\ge 2$, the functions $\sigma_i$ and $\sigma_j$ commute because they act on disjoint coordinates. Thus, by Lemma \ref{L:tm2} we have
		\[\sigma_m\omega=\tau_{n-1}\cdots \tau_{m+2}\sigma_m\tau_{m+1}\tau_m\cdots \tau_1=\tau_{n-1} \cdots\tau_{m+2}\tau^2_m\cdots \tau_1 = \tau_{n-1} \cdots\tau_{m+2}\tau_{m+1}\tau_m\cdots \tau_1=\omega. \qedhere
		\]
	\end{proof}
	\begin{proof}[Proof of Proposition \ref{P:W=w}] Part (i) follows immediately from Proposition \ref{P:wm-equivalent}. Part (ii) follows from Lemma \ref{L:w-longest}. 
	
	For (iii), Lemma \ref{L:idempotency} (iii) implies that $\omega$ is the identity on $\sD^n_\uparrow$, so $\omega(\sD^n)\supset \sD^n_\uparrow$. For the other direction, note that $\sigma_i$ does not change $\omega f$ by Lemma \ref{L:w-longest}. By 
	Lemma \ref{L:idempotency} (iv), this implies that $f_{i+1}\preceq f_{i}$ for all $i$, so $\omega(\sD^n)\subset \sD^n_\uparrow$.
	
	Finally, for (iv), $$Wf=W\omega f=\omega f,$$
where the first equality is a consequence of part (i), and the second is a consequence of \eqref{E:DuW}, since $\omega f\in \sD_\uparrow$ by (iii).
	\end{proof}
	%

	\begin{remark}
		We mention without proof that by applying a similar framework in the $n=3$ case, $\omega$ can alternately be defined as $\sig_1 \sig_2 \sig_1$, instead of $\sig_2 \sig_1 \sig_2$. This yields a braid relation for the operators $\sig_i$, which implies that for larger $n$, $W = \omega = \sig_{i_1} \cdots \sig_{i_{\binom{n}{2}}}$ whenever the product of transpositions $(i_1 \; i_1 +1)\cdots (i_{\binom{n}{2}} \; i_{\binom{n}{2}} + 1)$ is a reduced word for the reverse permutation. This was shown in \cite*{biane2005littelmann} in the case of continuous functions $f$.
	\end{remark}
	
\section{Minimal and maximal elements}
\label{S:max-element}

In this section, we show that the melon map picks out the minimal element in the equivalence class of isometric environments under a certain natural preorder. Recall that a preorder is partial order without the antisymmetry requirement, i.e. $f\le g$ and $g\le f$ are both allowed for $f\not=g$. 
While this perspective is not necessary for the proofs in later sections, it is helpful for developing intuition about the melon map and its inverse. 

For $f \in \sD^n$, let $\lambda_f(\bx, \by), \rho_f(\bx, \by)$ be the leftmost and rightmost optimizers in $f$ from $(\bx, n)$ to $(\by, 1)$, see Lemma \ref{L:leftmost-exist}.
We say that $f \geole g$ if:
\begin{equation*}
\lambda_f(\bx, \by) \le \lambda_{g}(\bx, \by)  \;\;\text{ and }\;\;  \rho_f(\bx, \by) \le \rho_{g}(\bx, \by)\;\;\text{ for all }\;\bx, \by.
\end{equation*}
Our goal is to prove the following. 
\begin{prop}
\label{P:min-elt-W}
For every $f \in \sD^n$ with $f(0) = 0$, the melon $Wf$ is the unique minimal element with respect to $\geole$ in the isometry class of $f$.
\end{prop}

\begin{remark}
The proof of Proposition \ref{P:min-elt-W} shows that $Wf \geole g$ for any $f, g \in \sD^n$ with $f \sim g$. The condition that $f(0) = 0$ is only necessary for the uniqueness statement. Indeed, consider the example where $f_i(x) = 1$ for all $x \ge 0$ so that $df$ consists of a line of atoms at $\{0\} \X \{1, \dots, n\}$. Then $Wf_1 = n$ and $Wf_i = 0$ for all $i \ge 2$ so $Wf \ne f$. However, it is not difficult to check that $\lambda_f(\bx, \by) = \lambda_{Wf}(\bx, \by), \rho_f(\bx, \by) = \rho_{Wf}(\bx, \by)$ for all $\bx, \by$ so we have both $f \geole Wf$ and $Wf \geole f$. 
\end{remark}

The main lemma needed for Proposition \ref{P:min-elt-W} concerns the Pitman transform $W:\sD^2 \to \sD^2$. Its proof is in a similar vein to the proof of Lemma \ref{L:wm-lem}. As a result, we defer it to the appendix where it is proven with Lemma \ref{L:wm-lem}.

\begin{lemma}
\label{L:push-back}
Let $f \in \sD^2$ and let $x \le y$. Then
$$
\rho_{Wf}(x, y) \le \rho_f(x, y) \quad \text{ and } \quad \lambda_{Wf}(x, y) \le \lambda_f(x, y).
$$
\end{lemma}

We can extend this to general optimizers using the same ideas as in Lemma \ref{L:n2-general}. As the proof is essentially identical, we omit it.

\begin{lemma}
\label{L:push-back-2}
For all $f \in \sD^2$, we have $Wf \geole f$.
\end{lemma}


We can extend this to $\sig_i$ by a simple application of the metric composition law. The proof is essentially identical to the proof of Proposition \ref{P:wm-equivalent} so we omit it.

\begin{lemma}
\label{L:push-back-3}
For all $f \in \sD^n$ and $i \in \{1, \dots, n-1\}$, we have $\sig_i f \geole f$.
\end{lemma}

\begin{proof}[Proof of Proposition \ref{P:min-elt-W}]
By Lemma \ref{L:push-back-3} and the definition, we have $Wf \geole f$ for all $f$. Now, for any $g$ in the same isometry class as $f$, we have $Wf = Wg \geole g$. Since $\geole$ is only a preorder, to prove that $Wf$ is the \emph{unique} minimal element we still need to check that if $f \geole Wf$ then $f = Wf$.

Indeed, since $Wf$ is Pitman ordered the leftmost optimizer $\lambda_{Wf}(0^k, x^k)$ simply follows the top $k$ paths $Wf_1, \dots, Wf_k$ on the interval $(0, x)$. Since no disjoint $k$-tuple is to the left of this optimizer, the same must be true for $\lambda_{f}(0^k, x^k)$ if $f \geole Wf$. Therefore for all $x, k$ we have 
$$
\sum_{i=1}^k Wf_i(x) = f[(0^k, n) \to (x^k, 1)] = \sum_{i=1}^k f_i(x) - f_i(0^-) + \sum_{i=k + 1}^n f_i(0) - f_i(0^-).
$$
Since $f(0^-) = 0$ by definition and $f(0) = 0$ by assumptions, this implies that $Wf = f$.
\end{proof}

Just as $Wf$ identifies a minimal element of the isometry class of $f$, we would also like to identify a maximal element. We can do this with a straightforward symmetric definition if we first restrict to functions defined on $[0, t]$ rather than $[0, \infty)$.

Let $\sD^n_t$ be the set of $n$-tuples of cadlag functions $f_1, \dots, f_n:[0, t] \to \R$ with only positive jumps. For $f \in \sD^n_t$, we can define the truncated melon map $W_t:\sD^n_t \to \sD^n_t$ by \eqref{E:melon-def-body}. Again, we can think of $W_t f$ as picking out a distinguished minimal element in the isometry class of $f$ in $\sD^n_t$. In the finite setting, a conjugation of the map $W_t$ produces a maximal element.
For $f \in \sD^n_t$, recall $R_t f \in \sD^n_t$ is the rotation of $f$ by 180 degrees. Define $M_t = R_t W_t R_t$. Then we have the following.
\begin{prop}
\label{P:max-elt}
For every $f \in \sD^n_t$ with $f(t^-) = f(t)$, $M_t f \in \sD^n_t$ is the unique maximal element in the isometry class of $f$ with respect to $\geole$.
\end{prop}

\begin{proof}
The main idea is that the reflection operator $R$ sends left most-geodesics to right-most geodesics and vice versa. For any $f$ and $\bx, \by$, we have
$$
\lambda_{Rf}(R \bx, R \by) = R[\rho_{f}(\bx, \by)] \quad \text{ and } \quad \rho_{Rf}(R \bx, R \by) = R[\lambda_{f}(\bx, \by)].
$$
Therefore
$$
\rho_{RW_t Rf}(\bx, \by) = R[\lambda_{W_t Rf}(R\bx, R\by)] \ge  R[\lambda_{Rf}(R\bx, R\by)] = \rho_{f}(\bx, \by),
$$
where the inequality uses that $W_t Rf \geole Rf$. A similar calculation for $\lambda_{RW_t Rf}$ shows that $f \geole M_t f$. Using that $M_t f = M_t g$ for all $g$ isometric to $f$ yields the that $g \geole M_t f$ for all $g \sim f$. Finally, uniqueness follows from the same reasoning as in the proof of Proposition \ref{P:min-elt-W}, but working with rightmost instead of leftmost optimizers.
\end{proof}

\section{The infinite bijection}
\label{S:bijectivity}
In this section we find the inverse of the melon map $W$. Let $\sD^n_t$ be the set of $n$-tuples of cadlag functions $f_1, \dots, f_n:[0, t] \to \R$ with only positive jumps. For $f \in \sD^n_t$, we can define the truncated melon map $W_t:\sD^n_t \to \sD^n_t$ by \eqref{E:melon-def-body}. 
Recall from the introduction \eqref{E:M-def-intro} that $M_t$ is the conjugate of $W_t$ by the 180 degree rotation $R_t$, namely $M_t=R_tW_tR_t$. By the definition of $W$, this can be written as
\begin{equation}
\label{E:Mtg-def}
M_tf_{n-k+1}(x^-)+\ldots + M_tf_n(x^-)= f[(0,n)^k\to(t,1)^k] - f[(x,n)^k\to (t,1)^k].
\end{equation}
For $f \in \sD^n$, we define the (infinite) \textbf{lemon map} by the limiting formula
\begin{equation}
\label{E:Mf-formula}
M f = \lim_{t \to \infty} M_t f.
\end{equation}
We note in passing that $M$ records a collection of multi-point Busemann functions for the metric environment defined by $f$. Before studying $M$, we must justify why the limit exists.

\begin{lemma}
	\label{L:M-exists}
For any $f \in \sD^n$, the limit \eqref{E:Mf-formula} exists in the topology of uniform convergence on compact sets. In particular, the map $M$ is well-defined from $\sD^n$ to $\sD^n$.
\end{lemma}

\begin{proof}
For $g \in \sD^n$, we let $\Sig_k g = g_{n-k+1} + \dots + g_n$. For every $k \in \{1, \dots, n\}$, $x_1 \le x_2 \le t_1 \le t_2 \in [0, \infty)$, we have
\begin{equation}
\label{E:Sigmono-2}
\Sig_k M_{t_2}f(x^-_1) - \Sig_k M_{t_1}f(x^-_1) \ge \Sig_k M_{t_2}f(x^-_2) - \Sig_k M_{t_1}f(x^-_2).
\end{equation}
This is a consequence of the definition \eqref{E:Mtg-def} and the quadrangle inequality (Lemma \ref{L:quadrangle}), with $\bp = (x_1, n)^k, \bp' = (x_2, n)^k, \bq = (t_1, 1)^k, \bq' = (t_2, 1)^k$. 

Plugging in $x_1 = 0$, the left hand side of \eqref{E:Sigmono-2} equals $0$. Thereforese see $\Sig_k M_t f(x_2^-)$ is nonincreasing in $t$ for every fixed $x_2$. Moreover, from \eqref{E:Mtg-def} we can conclude that
\begin{equation}
\label{E:SigkM}
\Sig_k M_{t} f(x^-) \ge \Sig_k f(x^-)
\end{equation}
since any disjoint $k$-tuple from $(x, n)^k$ to $(t, 1)^k$ can be extended to a disjoint $k$-tuple from $(0, n)^k$ to $(t, 1)^k$ by appending on the segments $[0, x) \X \{i\}, i = n, \dots, n - k + 1$. Combining these facts implies that there is some collection of functions $Mf = (M_1f, \dots, M_nf)$ such that $\Sig_k M_tf \cvgdown \Sig_k Mf$ pointwise for all $k$.

Next, we establish uniform convergence. Returning to \eqref{E:Sigmono-2}, we can replace $M_{t_2}$ with $M$, to get that $g_t(x) := \Sig_k Mf(x) -\Sig_k M_tf(x)$ is monotone in $x$ for every $t$. Moreover, we have already shown $g_t(x) \to 0$ pointwise in $x$. Any functions $g_t$ with these two properties must converge uniformly on compact sets to 0. This implies that $M_t f \to Mf$ uniformly on compact sets.

Finally, the space $\sD^n$ is closed in the topology of uniform-on-compact convergence so $M f \in \sD^n$, yielding the final part of the lemma.
\end{proof}
The definition of $M_t = R_t W_t R_t$ and Lemma \ref{L:M-exists} suggest that many properties of the melon map $W$ in Proposition \ref{P:W=w} should have parallels for the lemon map $M$. This is indeed the case.

\begin{lemma} The lemon map $M$ has the following properties.
	\label{L:M-properties-1}
\begin{enumerate}[label=(\roman*)]
	\item (Isometry) $Mg \sim g$.
	\item (Idempotence) $M^2 = M$.
\end{enumerate}
\end{lemma}

\begin{proof}
Part (i) for $M_t$ is immediate from the definition of $M_t = R_t W_t R_t$ and the corresponding property of $W_t$, Proposition \ref{P:W=w}(i). This property is closed under the limit in \eqref{E:Mf-formula} since last passage values are continuous in the uniform norm. Part (ii) follows from part (i), since $M$ is defined in terms of last passage values from line $n$ to line $1$.
\end{proof}



%

To define the image of $M$, we need the following notion.
First, we say that a $k$-tuple $\pi = (\pi_1, \dots, \pi_k)$ is a disjoint $k$-tuple from $(0, n)^k$ to $(\infty, 1)^k$ if $\pi$ is the jump-time limit of disjoint $k$-tuples $\pi^\ell$ from $(0, n)^k$ to $(\ell, 1)^k$ for some sequence $\ell \to \infty$.  Jump-time limit means the jump times \eqref{E:jumptimes} of each individual path of the $k$-tuple $\pi^\ell$ converge to the jump times of each individual path in $\pi$. The point $+\infty$ is considered a valid limit; in this case the limiting path will not intersect the top line. 

Notions of left and right naturally extend to these $k$-tuples of semi-infinite paths.
We say that $\pi$ is an \textbf{(infinite) quasi-optimizer} for $g$ if the $\pi^\ell$ can be chosen so that
\begin{equation}
\label{E:ell-to-infty}
\lim_{\ell \to \infty} |\pi^\ell|_g - g[(0, n)^k \to (\ell, n)^k] = 0,
\end{equation}
and $\pi$ is an \textbf{(infinite) optimizer} if the $\pi^\ell$ can be chosen so that $|\pi^\ell|_g = g[(0, n)^k \to (\ell, n)^k]$ for all large enough $\ell$. For every $k \in \{1, \dots, n\}$, by taking a subsequential limit of a sequence of optimizers from $(0, n)^k \to (\ell, n)^k$, there is at least one optimizer from $(0, n)^k$ to $(\infty, 1)^k$.


Let $\sD^n_\downarrow\subset \sD^n$ be the space where the \textbf{rightmost} $k$-tuple from $(0, n)^k$ to $(\infty, 1)^k$ is a \textbf{quasi-optimizer} for all $k$. Similarly, let $\sD^n_{\downarrow -}$ be the space where the rightmost $k$-tuple from $(0, n)^k$ to $(\infty, 1)^k$ is an \textbf{optimizer}
for all $k$. 

The image $\sD^n_\uparrow$ of $W$ can be thought of as the space where the \textbf{leftmost} $k$-tuple from $(0, n)^k$ to $(\infty, 1)^k$ is an \textbf{optimizer} for all $k$. The analogue for $M$ is given by the following two lemmas.

\begin{lemma}
	\label{L:M-properties-3}
$M(\sD^n) = \sD^n_\downarrow$.
\end{lemma}

\begin{proof}
Since $M$ idempotent by Lemma \ref{L:M-properties-1}(ii), $g$ is in its image if and only if $Mg = g$. We use the notation $\Sig_k g = g_{n-k+1} + \dots + g_n$. By \eqref{E:Mtg-def}, $Mg = g$ if and only if
\begin{equation}
\label{E:f0x}
\lim_{t \to \infty} g[(0,n)^k\to(t,1)^k] - g[(x,n)^k\to (t,1)^k] - \Sig_k g(x^-) = 0
\end{equation}
for every $x, k$. We now prove that these $g$ are exactly $\sD^n_{\downarrow}$ by proving both inclusions.

\medskip

{\noindent \bf Only if: $Mg=g \Rightarrow g\in \sD^n_{\downarrow}$.} If \eqref{E:f0x} holds for every $x, k$, then by a diagonalization argument we can find $x_\ell \to \infty$, $x_\ell\le \ell$ with $ \ell \in \N$ such that
\begin{equation}
\label{E:f0xx}
\lim_{\ell \to \infty} g[(0,n)^k\to(\ell,1)^k] - g[(x_\ell,n)^k\to (\ell,1)^k] - \Sig_k g(x_\ell^-) = 0
\end{equation}
for every $k$. For $\ell \in \N$, let $\pi^\ell$ be the concatenation of the bottom $k$ paths on the interval $[0, x_\ell)$ with an optimizer from
$(x_\ell,n)^k\to (\ell,1)^k$. We have $|\pi|_g = g[(x_\ell,n)^k\to (\ell,1)^k] + \Sig_k g(x_\ell^-)$, so by \eqref{E:f0xx}, the equation \eqref{E:ell-to-infty} is satisfied for this sequence $\pi^\ell$. Moreover, since $x_\ell \to \infty$ with $\ell$ and $\pi^\ell$ just follows by the bottom $k$ paths up time $x_\ell$, it converges to the rightmost $k$-tuple $\pi_k$ from $(0, n)^k$ to $(\infty, 1)^k$. Therefore $\pi_k$ is a quasi-optimizer, so $g\in \sD^n_\downarrow$.

\medskip

{\noindent \bf If: $ g\in \sD^n_{\downarrow} \Rightarrow Mg=g $.} If the rightmost $k$-tuple $\pi_k$ from $(0, n)^k$ to $(\infty, 1)^k$ is a quasi-optimizer, then, by the definition of $\sD^n_{\downarrow}$, there is a sequence $\ell \to \infty$ and a sequence of disjoint $k$-tuples $\pi^\ell$ from $(0, n)^k$ to $(\ell, n)^k$ that converge to $\pi_k$ and satisfy \eqref{E:ell-to-infty}. Since $\pi^\ell \to \pi_k$, the optimizer $\pi^\ell$ uses the bottom $k$ paths up to some time $x_\ell \to \infty.$ This implies \eqref{E:f0x} for any fixed $x$ as long as we take the limit only over the sequence $\ell$, rather than over all $t \in [0, \infty)$. We can pass to the full limit in $t$ since we know this limit exists by Lemma \ref{L:M-exists}.
\end{proof}

\begin{lemma}
	\label{L:M-properties-2}
$\sD^n_\downarrow = \overline{\sD^n_{\downarrow -}}$, where the closure is taken in the uniform norm.
\end{lemma}
\begin{proof}
{\noindent \bf Inclusion $\sD^n_{\downarrow} \subset \overline{\sD^n_{\downarrow -}}$.} Let $g \in \sD^n_\downarrow$, and for $\ep > 0$, define $g^\ep_i(t) = g_i(t) + \ep i \arctan (t)$, so that $g^\ep \to g$ uniformly as $\ep \to 0$. This approximation $g^\ep$ adds more weight onto lower lines $n, n-1, \dots$, encouraging optimizers to use these lines. Specifically, if a path $\tau$ from $p$ to $q$ is to the left of a path $\pi$ from $p$ to $q$ then \begin{equation} \label{eq:epsilon_environment_property}
    \tau|_{g^\epsilon} - |\tau|_g < |\pi|_{g^\epsilon} - |\pi|_g.
\end{equation}We claim that $g^\ep \in \sD^n_{\downarrow -}$ for all $\ep > 0$. 

Fix $k$ and let $\pi^\ell$ be the sequence in \eqref{E:ell-to-infty} for $g$ converging to the rightmost $k$-tuple $\pi_k$ from $(0, n)^k$ to $(\infty, 1)^k$. Let $\tau^\ell$ be any sequence of optimizers in $g^\ep$ from $(0, n)^k$ to $(\ell, n)^k$. Suppose, for the sake of contradiction, that $\tau^\ell \not\to \pi_k$, and by choosing a subsequence if necessary, suppose $\tau^\ell \to \tau \ne \pi_k$. Since paths in $\tau$ are to the left of paths in $\pi_k$, \eqref{eq:epsilon_environment_property} guarantees that there exists $\de > 0$ such that for all large enough $\ell$, we have
$$
\de + |\tau^\ell|_{g^\ep} - |\tau^\ell|_{g} < |\pi^\ell|_{g^\ep} - |\pi^\ell|_{g}.
$$
On the other hand, the liminf of $|\pi^\ell|_{g} - |\tau^\ell|_{g}$ is at least $0$ by the quasi-optimality of $\pi^\ell$, implying that $|\pi^\ell|_{g^\ep} > |\tau^\ell|_{g^\ep} + \de$ for large enough $\ell$, contradicting the optimality of $\tau^\ell$. Hence $\tau^\ell \to \pi_k$, and so $g^\ep \in \sD^n_{\downarrow -}$.

\medskip

{\noindent \bf Inclusion $\sD^n_{\downarrow} \supset \overline{\sD^n_{\downarrow -}}$.} As in the previous argument, let $\pi_k$ be the rightmost $k$-tuple from $(0,n)^k$ to $(\infty,1)^k$. Suppose $g\in \overline{\sD^n_{\downarrow -}}$, and suppose $g_m \in D^n_{\downarrow -}$ converges uniformly to $g$. Fix $k \in \{1, \dots, n\}$, and for each $m$, let $\pi^\ell_m$ be a sequence of disjoint optimizers in $g_m$ from $(0, n)^k$ to $(\ell, 1)^k, \ell \in \N$ that converges along a subsequence to $\pi_k$. By a diagonalization argument, we can find a sequence $\ell_m \to \infty$ such that $\pi^{\ell_m}_m \to \pi_k$ as $m \to \infty$ and $|\pi^{\ell_m}_m|_{g_m} - |\pi^{\ell_m}_m|_{g} \to 0$. The sequence $\pi^{\ell_m}_m$ is quasi-optimal for $g$, so $g \in \sD^n_{\downarrow}.$
\end{proof}

Later, we will also use the following closely related observation.
\begin{corollary}
	\label{C:Mf0t}
	Let $\sD^n_{*t}$ be the set of $f \in \sD^n$ such that for every $k$ there is a quasi-optimizer $\tau_k$ from $(0, n)^k \to (\infty, 1)^k$ that agrees with the rightmost $k$-tuple from $(0, n)^k \to (\infty, 1)^k$ up to time $t$. For $f \in \sD^n_{*t}$, we have
	$
	Mf_{[0, t]} = f|_{[0, t]}.
	$
\end{corollary}

\begin{proof}
Just as in the `if' part of the proof of Lemma \ref{L:M-properties-1}, the existence of such a quasi-optimizer $\tau_k$ implies \eqref{E:f0x} for $x \le t$, giving the result.	
\end{proof}
 
 The fact that $M$ and $W$ are inverses follows from an abstract lemma.
 
 \begin{lemma}
 \label{L:idempotent-inverse}
 Let $A, B:X \to X$ be two idempotent maps satisfying $AB = A$ and $BA = B$. Then $A|_{B(X)}$ is a bijection between $B(X)$ and $A(X)$ with inverse $B|_{A(X)}$.
 \end{lemma}

\begin{proof}
Since $B$ is idempotent, $B$ is the identity on $B(X)$. Therefore $BA = B$ is also the identity on $B(X)$. Similarly, $AB$ is the identity on $A(X)$. Finally, $A(X) = AB(X)$ and $B(X) = BA(X)$, yielding the result.
\end{proof}

\begin{prop}
	\label{P:MW-relations}
We have $MW = M$ and $WM = W$. Moreover, the restriction $W|_{\sD^n_\downarrow}$ is a bijection between $\sD^n_\downarrow$ and  $\sD^n_\uparrow$ with inverse $M|_{\sD^n_\uparrow}$.
\end{prop}

\begin{proof}
The first sentence is immediate from the isometric properties of $W$ and $M$ and the fact that both maps are defined in terms of last passage from line $n$ to line $1$. The second sentence follows from Lemma \ref{L:idempotent-inverse}.
\end{proof}

It may seem that elements of $\sD^n_{\downarrow}$ in $\sD^n$ should be somewhat rare and special. Surprisingly, this is not the case! The following proposition gives a natural condition for an element of $\sD^n$ to be in $\sD^n_\downarrow$ (in fact, in $\sD^n_{\downarrow -}$).

\begin{prop}
\label{P:sourness}
Let $f \in \sD^n$, and suppose that
\begin{equation}
\label{E:limsup-infty}
\limsup_{t \to \infty} f_{j+1}(t) - f_j(t) = \infty
\end{equation}
for $j = 1, \dots, n-1$. Then $f\in \sD^n_{\downarrow -}$.
\end{prop}

\begin{proof}
Suppose for the sake of contradiction that $f \notin \sD^n_{\downarrow-}$. Then for some $k \in \{1, \dots, n\}$, there is an optimizer $\pi = (\pi_1, \dots, \pi_k)$ from $(0, n)^k$ to $(\infty, 1)^k$ such that $\pi_i \ne [0, \infty) \X \{n + k - i\}$ for some $i$. Let $i$ be the maximal index for which this holds, let $j < n + k - i$ be the index of the line on which $\pi_i$ contains a semi-infinite interval, and let $t \in [0, \infty)$ be the time when $\pi_i$ jumps to line $j$, so that $(t, j+1) \in \pi_i$ and $[t,  \infty) \X \{j\} \sset \pi_i$. Since $\pi_\ell = [0, \infty) \X \{n + k - \ell\}$ for $\ell < i$, we have that 
\begin{equation}
\label{E:tjempty}
(t, \infty) \X \{j+1\} \cap (\cup \pi) = \emptyset.
\end{equation}
Since $\pi$ is an optimizer, for any $s > t$, restricting each of the $\pi_i$ to a compact set $[t, s] \X \{1, \dots, n\}$ must also yield an optimizer. In particular, \eqref{E:tjempty} implies that the path $\pi_i \cap [t, s] \X \{j, j+1\}$ must always be a geodesic from $(t, j+1)$ to $(s, j)$. Therefore
\begin{equation}
\label{eq:sup_j_j1}
\sup_{r \in [t, s]} f_{j+1}(r) -f_j(r^-) = \sup_{r \in [t, s]} (f_{j+1}- f_j)(r) + [f_j(r) - f(r^-)] = f_{j+1}(t) - f_j(t^-) 
\end{equation}
for all $s > t$. Taking $s \to \infty$ in \eqref{eq:sup_j_j1} shows that the left hand side of \eqref{E:limsup-infty} equals $f_{j+1}(t) - f_j(t^-) < \infty$, a  contradiction of  \eqref{E:limsup-infty}.
\end{proof}

\begin{remark}
\label{R:sourness-of-random-walks}
Natural measures on $\sD^n$, 
such as i.i.d.\ random walks or i.i.d.\ L\'evy processes, satisfy the conditions of Proposition \ref{P:sourness} almost surely. Many of these measures, such as Brownian motion, piecewise constant geometric random walks, or piecewise linear versions of Bernoulli walks, appear in integrable models. See Section \ref{S:preservation}. 
\end{remark}

\section{The finite bijection}
\label{SS:finite-bijection}
In Section \ref{S:bijectivity}, we studied the bijectivity of the melon map in the infinite-time setting. The goal of this section is to construct a version in the finite-time setting. This version is more similar to classical RSK:  an additional Gelfand-Tsetlin pattern will play the role of the  second Young tableaux.  

Like the infinite case, it is straightforward to check that the maps $W_t$ and $M_t$ are inverses of each other on appropriate domains. However, unlike the infinite case, in the finite case, the image of $M_t$ is a rare subset of $\sD^n_t$ and so the bijection loses usefulness. 
Instead, our strategy is to take an element of $\sD^n_t$ and map it to an element of $\sD^n_{*t}$, see Corollary \ref{C:Mf0t}. This allows us to define the RSK correspondence on the finite domain $[0,t]$ in terms of the maps $W,M$ defined to the infinite domain $[0,\infty)$.


For $f \in \sD^n_t$, and a $k$-tuple of functions $g = (g_1, \dots, g_n)$ where each $g_i:[t, \infty) \to \R$ and $g(t) = f(t)$, define the \textbf{concatenation} $f \oplus g \in \sD^n$ to be equal to $f$ on $[0, t]$ and $g$ on $[t, \infty)$. Note that our notation implicitly depends on $t$. For $\al > 0$, we write $f \oplus \al$ to mean $f \oplus g_\al$, where $d g_\al$ is purely atomic on $(t, \infty) \X \{1, \dots, n\}$ with atoms of size $\al i$ at locations $(t + i, i)$. The concatenations $f \oplus \al$ are set up so that for large enough $\al$, $f \oplus \al \in \sD^n_{*t}$, see Figure \ref{fig:pathsalpha}.

Given this, one direction of the finite RSK correspondence will essentially be the map $f \mapsto W(f \oplus \al)$ for large $\al$. We can then invert this correspondence using $M$:
\begin{equation}
\label{E:MWf}
M W(f \oplus \al)|_{[0, t]} = M(f \oplus \al)|_{[0, t]} = (f \oplus \al)|_{[0, t]} = f.
\end{equation}
Here the first equality follows from Proposition \ref{P:MW-relations}, the second follows from Corollary \ref{C:Mf0t}, and the third is by definition. 

While \eqref{E:MWf} clearly hints at the existence of a bijection, there are a few more things to check. We need to figure out what information about $f$ is actually used in the map $f \mapsto W(f \oplus \al)$, identify the image of this map, and then show that if we start with data in that space, then $WM$ is also the identity on that space. The rest of this section is devoted to doing this. 
We will also present our version of the finite RSK correspondence in a way that more closely resembles the classical RSK bijection between matrices and Young tableaux. Our correspondence will be related to classical RSK and dual RSK in Section \ref{S:embedding}.

%
%
%
%
%

First, let $\sD^n_{t \uparrow} = W_t(\sD^n_t)$ be the space of Pitman ordered sequences $w_n \preceq \dots \preceq w_1$ in $\sD^n_t$. For $n \in \N$, let $\operatorname{GT}_n$ be the space of triangular arrays $g = \{g_i(j) : i \le j \in \{1, \dots, n\}\}$ satisfying the inequalities
\begin{equation}
\label{E:GT-inequalities}
g_i(j) \ge g_i(j-1), \qquad g_i(j-1) \ge g_{i+1}(j),
\end{equation}
for all $i, j$ where these quantities are defined. Such an array is called a \textbf{Gelfand-Tsetlin pattern} of depth $n$. These inequalities amount to saying that the $j$th row $g(j)$ is an ordered $i$-tuple, and that consecutive rows interlace.
%
%
Next, define the space
$$
\sG^n_t = \{ h = (w, g) \in \sD^n_{t \uparrow} \X \operatorname{GT}_n : w(t) = g(n) \}.
$$
The space $\sG^n_t$ is the analogue of the set of pairs of Young tableaux with the same shape, see Section \ref{S:embedding}.
The finite-$t$ RSK correspondence, $\operatorname{RSK}_t$, maps $\sD^n_t$ into $\sG^n_t$, where $\operatorname{RSK}_t(f) = (W_t f, G_t f)$, and for $k \le s \in \{1, \dots, n\}$ we have
\begin{equation}
\label{E:Gtf-formula}
\sum_{i=1}^k G_t f_i(s) = f[(0, n)^{k} \to (t, n - s + 1)^{k}].
\end{equation} 
\begin{lemma}
	\label{L:RSK-well-defined}
	The map $\operatorname{RSK}_t:\sD^n_t \to \sG^n_t$ is well-defined. That is, $\operatorname{RSK}_t f \in \sG^n_t$ for any $f \in \sD^n_t$.
\end{lemma}

\begin{proof}
	By Proposition \ref{P:W=w} (iii), $W_t f \in \sD^n_{t \uparrow}$, and $W_tf(t) = G_tf(n)$ by definition. It remains to verify the Gelfand-Tsetlin inequalities \eqref{E:GT-inequalities} for $G_t f$. Both of these types of inequalities have similar proofs and follow from quadrangle inequalities that are similar in spirit to Lemma \ref{L:quadrangle}. We prove only the first inequality; the second one is similar. By \eqref{E:Gtf-formula}, the inequality $G_t f_i(s) \ge G_t f_i(s-1)$ is equivalent to the quadrangle inequality
	\begin{equation}
	\label{E:quad-interior}
	\begin{split}
&f[(0, n)^{i} \to (t,n - s + 1)^{i}] + f[(0, n)^{i-1} \to (t,n - s + 2)^{i-1}] \\
\ge \;\; &f[(0, n)^{i} \to (t,n - s + 2)^{i}] + f[(0, n)^{i-1} \to (t, n - s + 1)^{i-1}].
\end{split}
	\end{equation}
To prove \eqref{E:quad-interior}, let $\pi = (\pi_1, \dots, \pi_i), \pi' = (\pi'_1, \dots, \pi'_{i-1})$ be disjoint optimizers from $(0, n)^{i} \to (t,n - s + 2)^{i}$ and $(0, n)^{i-1} \to (t, n - s + 2)^{i-1}$, respectively. Similarly to the proof of Lemma \ref{L:quadrangle}, we will use $\pi, \pi'$ to construct a disjoint $i$-tuple $\tau^L$ from $(0, n)^{i} \to (t,n - s + 1)^{i}$ and a disjoint $(i-1)$-tuple $\tau^R$ from 
$(0, n)^{i-1} \to (t,n - s + 2)^{i-1}$. $\tau^L$ and $\tau^R$ will be constructed by dividing up $\pi$ and $\pi^\prime$ in such a way that 
\begin{equation}\label{eq:tau_pi_party}
|\pi|_f + |\pi'|_f = |\tau^L|_f + |\tau^R|_f.
\end{equation}
For $j \le i-1$, we let $\tau^L_{j}$ be the leftmost path from $(0, n)$ to $(t,n - s + 1)$ contained in the union $\pi_j \cup \pi_j'$, and let $\tau^L_i = \pi_i \cup \{(t, n - s + 1)\}$. Also, for $j \le i-1$ let $\tau^R_j$ be the rightmost path from $(0, n)$ to $(t,n - s + 2)$ contained in the union $\pi_j \cup \pi_j'$. With this construction, $\tau^L$ is a disjoint $i$-tuple from $(0, n)^{i} \to (t,n - s + 1)^{i}$, $\tau^R$ is a disjoint $(i-1)$-tuple from $(0, n)^{i-1} \to (t,n - s + 2)^{i-1}$, and from this definition we see that \eqref{eq:tau_pi_party} holds.
	
 Now, $|\tau^L|_f + |\tau^R|_f$ is bounded above by the left side of \eqref{E:quad-interior} and $|\pi|_f + |\pi'|_f$ equals the right side of \eqref{E:quad-interior}, yielding the result.
\end{proof}

We now move to understanding the map $f \mapsto W(f \oplus \al)$ for large $\al$.

\begin{lemma}
	\label{L:f-lemma}
	Let $f \in \sD^n_t$. For $\al \ge (n-1)(Wf_1(t) - Wf_n(t))$, we have $f \oplus \al \in \sD^n_{*t}$ and $W(f \oplus \al) = W_t f \oplus \Delta_{t, \al} G_t f$, for some function $\Delta_{t, \al}:\operatorname{GT}_n \to \sE^{n, t}$. Here $\sE^{n, t}$ is the space of $k$-tuples of cadlag paths from $[t, \infty) \to \R$ (with possibly negative jumps).
\end{lemma}

\begin{proof}
On $[0, t]$, the formula in Proposition \ref{P:W=w}(iv) implies that $W(f \oplus \al) = Wf$. Next, we show that $W(f \oplus \al) $ is an explicit function of $G_t f$ on $[t, \infty)$. The action of $W$ applied to $f \oplus \alpha$  is illustrated in Figure \ref{fig:pathsalpha}. Indeed, we can see that for $k, \ell \in \{1, \dots, n\}$, there will be a disjoint $k$-tuple $\pi$ across $f \oplus \al$ from $(0, n)^\ell$ to $(t + k, 1)^\ell$ that picks up large $\al$-weights at locations $(1 + (k-\ell)^+, t + (k-\ell)^+ + 1), \dots, (k, t + k)$. By ensuring that these weights are chosen and that $\pi$ is also optimal on $[0, t]$, we can ensure that $\pi$ has weight 
$$
\sum_{i= 1 + (k - \ell)^+}^k \al i + f[(0, n)^\ell \to (t, (k-\ell)^+ + 1)^\ell] = \sum_{i= 1 + (k - \ell)^+}^k \al i + \sum_{i=1}^\ell G_t f_i(n - (k-\ell)^+).
$$
No other disjoint $k$-tuple from $(0, n)^\ell$ to $(1, t + k)^\ell$ can improve the $\al$-part of the sum above, and any disjoint $k$-tuple $\tau$ can only improve the $f$-part of the sum by at most
$$
f[(0, n)^\ell \to (t, 1)^\ell] - f[(0, n)^\ell \to(t, (k-\ell)^+ + 1)^\ell] \le (n-1)(Wf_1(t) - Wf_n(t)),
$$
at the expense of at least one $\al$. Since $\al > (n-1)(Wf_1(t) - Wf_n(t))$, we have $|\pi|_{f \oplus \al} > |\tau|_{f \oplus \al}$, so $\pi$ is an optimizer. The story for optimizers up to time $t + k$ when $k \notin \{1, \dots, n\}$ is similar by rounding down to the nearest integer time. Therefore $W(f \oplus \al) $ is an explicit function of $G_t f$ on $[t, \infty)$. Moreover, from the construction of optimizers above from $(0, n)^\ell$ to $(t+k, 1)^\ell$ for $k \ge n$, we can see that $f \oplus \al \in \sD^n_{*t}$. 
\end{proof}

The proof of Lemma \ref{L:f-lemma} allows us to explicitly construct the function $\Delta_{t, \al}$. While this will be necessary to fill in some proof details regarding the invertibility of $\operatorname{RSK}_t$, for now it will be easier to think of the map $\Delta_{t, \al}$ abstractly, as a \emph{linear map} from $\R^{\binom{n + 1}{2}}$ (which contains $\operatorname{GT}_n$) onto an $\binom{n + 1}{2}$-dimensional linear subspace of $\sE^{n, t}$, the space of $k$-tuples of cadlag functions from $[t, \infty) \to \R$. It has the following properties. 

\begin{lemma}
	\label{L:Delta-al-lemma} For every $t, \al > 0$, the map $\Delta_{t, \al}:\operatorname{GT}_n \to \sE^{n, t}$ satisfies the following:
	\begin{enumerate}[nosep, label=(\roman*)]
		\item $\Delta_{t, \al}$ is one-to-one.
		\item $\Delta_{t, \al}g(t) = g(n)$.
		\item For $\al \ge \al_g := (n-1)(g_1(n) - g_n(n))$, paths in $\Delta_{t, \al} g$ are cadlag with only positive jumps
		\item The paths in $\Delta_{t, \al} g$ are Pitman ordered: $\Delta_{t, \al} g_{i+1}(s^-) \le \Delta_{t, \al} g_i(s)$ for all $s, i$.
		\item For $\al \ge \al_g$ and $0 < r \le n$ we have
		\begin{equation}
		\label{E:Delta-mass}
		\Delta_{t, \al} g[(t + r, n)^k \to (t + n, 1)^k] = \sum_{i=\cl{r} \wedge (n - k + 1)}^n \al i.
		\end{equation}
	\end{enumerate}
\end{lemma}

Note that properties (ii)-(v) above are easy to see when $g = G_t f$ for some $f$. In this case, properties (ii)-(iv) follow from the fact that $Wf  \oplus \Delta_{t, \al} g = W(f \oplus \al) \in \sD^n_\uparrow$ (Lemma \ref{L:f-lemma}) and property (v) follows immediately from the fact that $Wf  \oplus \Delta_{t, \al} g$ is isometric to $f \oplus \al$. 

We postpone the proof of Lemma \ref{L:Delta-al-lemma} for now, and use it to show the invertibility of $\operatorname{RSK}_t$. Let $O_t: \sG^n_t \to \sD^n$ be the map taking a pair $(w, g) \mapsto w \oplus \Delta_{t, \al_g} g$, 
	where $\al_g = (n-1)(g_1(n) - g_n(n))$. The concatenation here is well-defined by Lemma \ref{L:Delta-al-lemma}(ii) and the fact that $w(t) = g(n)$.
	Also let $\Ga_t : \sD^n \to \sD^n_t$ be the restriction of a function to $[0, t]$.

\begin{figure}
\begin{center}
	\includegraphics[scale=1.2]{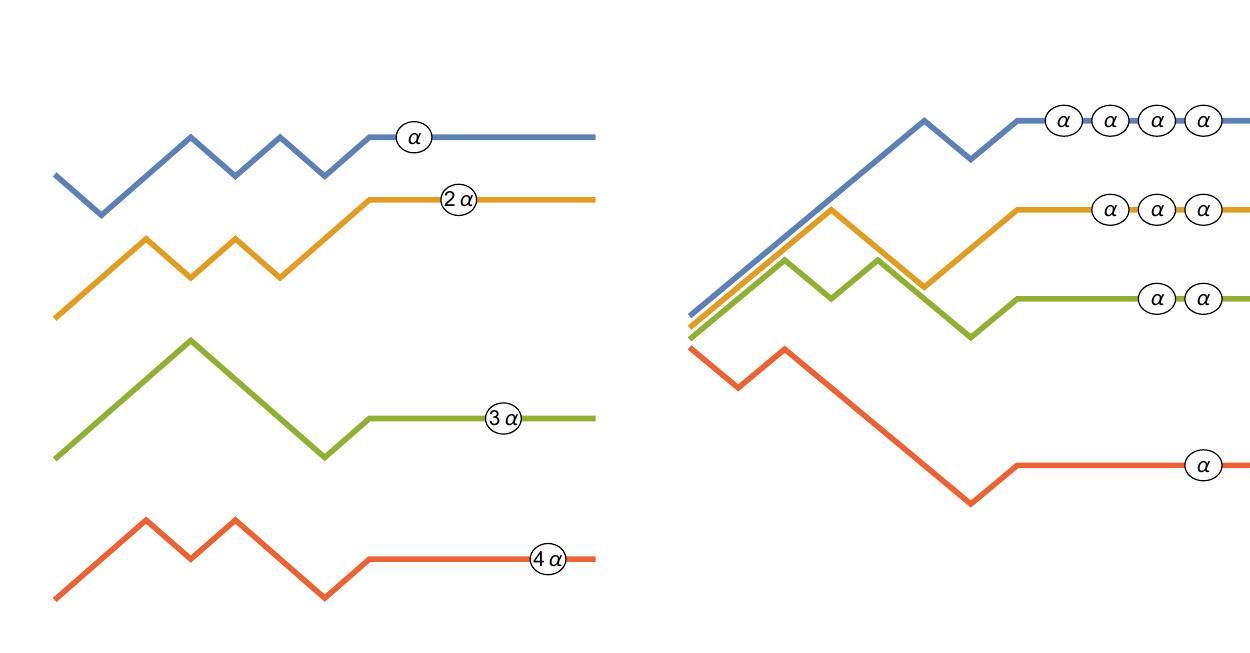}
\end{center}
\caption{An example of $f \oplus \al$ and $W(f \oplus \al)$ for a function $f \in \sD^4_t$. 
	In $W(f \oplus \al)$, the ten weights all have a contribution of a single $\al$ plus a contribution related to $G_t f$, not contained in the figure.}
	\label{fig:pathsalpha}
\end{figure}

\begin{prop}
	The map $\operatorname{RSK}_t:\sD^n_t \to \sG^n_t$ is a bijection with inverse
	$
	\operatorname{RSK}_t^{-1} := \Ga_t M O_t.
	$
\end{prop}

\begin{proof}
	First, for $(w, g) \in \sG^n_t$ we claim that
	\begin{equation}
	\label{E:MW}
	M (w \oplus \Delta_{t, \al_g} g) = f \oplus \al_g,
	\end{equation}
	for some $f \in \sD^n_t$. Indeed, from the definition \eqref{E:Mtg-def} and \eqref{E:Mf-formula} for $M$ we see that for any $0 < r \le n$ we have
	$$
	\sum_{i=n - k + 1}^n M(w \oplus \Delta_{t, \al_g} g)_i((t + r)^-)  = c_k -\Delta_{t, \al} g[(t + r, n)^k \to (t + n, 1)^k] 
	$$
	where $c_k = (w \oplus \Delta_{t, \al_g} g)[(0, n)^k \to (t + n, 1)^k]$. Next, Lemma \ref{L:Delta-al-lemma}(iii) gives that
	$$
	\Delta_{t, \al} g[(t + r, n)^k \to (t + n, 1)^k] = \sum^n_{\lceil r \rceil \wedge (n - k + 1)} \al i .
	$$
	This implies the representation \eqref{E:MW}. Next, by Proposition \ref{P:MW-relations}, 
	$$
	w \oplus \Delta_{t, \al_g} g = W M (w \oplus \Delta_{t, \al_g} g) = W(f \oplus \al_g).
	$$
	In particular, $W_t f = w$. Moreover,  $W_tf(t) = \Delta_{t, \al_g} g = g(n)$ by Lemma \ref{L:Delta-al-lemma}(ii), so $\al_g = (n-1)(Wf_1(t) - Wf_n(t))$. Therefore by Lemma \ref{L:f-lemma}, $W(f \oplus \al_g) = W_t f \oplus \Delta_{t, \al_g} G_t f$, so $\Delta_{t, \al_g} G_t f = \Delta_{t, \al_g} g$. Since $\Delta_{t, \al_g}$ is one-to-one (Lemma \ref{L:Delta-al-lemma}(i)), we get that $G_t f = g$. Putting all this together gives that $\operatorname{RSK}_t \Ga_t M O_t (w, g) = (w, g)$.
	
	On the other hand, $\Ga_t M O_t \operatorname{RSK}_t f = f$ via the computation \eqref{E:MWf}.
\end{proof}

\begin{proof}[Proof of Lemma \ref{L:Delta-al-lemma}]
	We first find an explicit formula for $\Delta_{t,\al}$. Following from the proof of Lemma \ref{L:f-lemma}, we can see that $\Delta_{t, \al}$ is given by the following two rules:
\begin{itemize}
	\item $\Delta_{t, \al}g(t) = g(n)$, and on $(t, \infty) \X \{1, \dots, n\}$, the finitely additive measure $d\Delta_{t, \al} g$ is purely atomic with support contained in the set of points $(t + k, \ell), \ell \le k \in \{1, \dots, n\}$.
	\item For $k, \ell \in \{1, \dots, n\}$,
	\begin{equation}\label{eq:g_alpha}
	\sum_{i=1}^\ell \Delta_{t, \al}g_i(t + k) = \sum_{i= 1 + (k - \ell)^+}^k \al i + \sum_{i=1}^\ell g_i(n - (k - \ell)^+).
	\end{equation}
\end{itemize}
{\bf Proof of (i, ii).} By rearranging equation \eqref{eq:g_alpha} it is verified from these formulas that $\Delta_{t, \al}g$ determines $g$, giving (i). Part (ii) follows from the first bullet point.

\noindent {\bf Proof of (iii).} Now, for $k \in \{1, \dots, n\}$ and $\ell < k$, from equation \eqref{eq:g_alpha}, $d\Delta_{t, \al} g$ has an atom of size
\begin{equation}
\label{E:alpha-atom}
\al - 
\sum_{i=1}^\ell [g_i(n - k + \ell + 1) - g_i(n - k + \ell)] + \sum_{i=1}^{\ell-1} [g_i(n - k + \ell)- g_i(n - k + \ell - 1)]
\end{equation} 
at $(t + k, \ell)$, and has an atom of size $\al + \sum_{i=1}^{k-1} g_i(n) - g_i(n-1)$ at the points $(t + k, k)$ for $k \in \{ 1, \dots, n\}$. Note that for $\ell = 1$ the second sum in \eqref{E:alpha-atom} is empty. To prove (ii), we first show that all these atoms have positive weight. For the atoms at $(t + k, k), k = 2, \dots, n$, this is true by the Gelfand-Tsetlin inequalities \eqref{E:GT-inequalities}. For the remaining atoms, by \eqref{E:GT-inequalities}, all terms in both sums in \eqref{E:alpha-atom} are nonnegative and bounded above by $g_1(n) - g_n(n)$. Since at most $n-1$ terms in \eqref{E:alpha-atom} come with a negative sign. By our choice of $\al_g$ and since $\al \ge \al_g$, these remaining atoms are positive. This shows that the paths in $\Delta_{t, \al} g$ have positive jumps.

\noindent {\bf Proof of (iv).} Next, we check that $\Delta_{t, \al} g_{i+1}(s) \le \Delta_{t, \al} g_{i}(s^-)$ for all $i \le n-1, s \in [0, \infty)$. This inequality holds at time $t$ since $g(n)$ is ordered by \eqref{E:GT-inequalities}. For $s > t$, it suffices to show that for $k, \ell \in \{1, \dots, n\}$,
$$
\Delta_{t, \al} g_{\ell+1}(t + k) \le \Delta_{t, \al} g_\ell(t + k -1).
$$ 
Writing out both sides of the inequality in terms of $g$ and cancelling common terms, this inequality is equivalent to
\begin{equation*}
g_{\ell+ 1}(n-(k -\ell -1)^+) \le g_{\ell}(n- (k - \ell)^+).
\end{equation*}
For $\ell < k$, this is the second Gelfand-Tsetlin inequality in \eqref{E:GT-inequalities}. For $\ell \ge k$, this follows since the vector $g(n)$ is ordered. 

\noindent {\bf Proof of (v).} (It may help with visualizing the argument to look at the right side of Figure \ref{fig:pathsalpha} when following the proof.) Let $0 <  r \le n$. First, for $i = 1, \dots, k$, let $\pi_i$ be the unique path from $(t + r, n)$ to $(t+n, 1)$ containing the sets 
$$
[t + (i + n - k) \vee r , t + n] \X \{i\}, \qquad [t + r, t + (i + n - k) \vee r] \X \{i + n - k\}.
$$ 
Then $\pi = (\pi_1, \dots, \pi_k)$ is a disjoint $k$-tuple from $(t+r, n)^k$ to $(t+n, 1)^k$ whose weight equal to $\al (\sum_{i=\cl{r}}^n i)$. For $k > n - \cl{r}$, this $k$-tuple collects all the atoms of $d \Delta_{t, \al} g$ in the box $[t + r, t+n] \X \{1, \dots, n\}$, and so must be an optimizer. 

For $k \le n - \cl{r}$, we need to check that no other disjoint $k$-tuple can have greater weight. First, since all atoms of $d \Delta_{t, \al} g$ are at times in $t + \Z$, we can restrict our attention to disjoint $k$-tuples from $(t+r, n)^k$ to $(t+n, 1)^k$ whose jumps are also in $t + \Z$.
Let $S$ be the finite set of such $k$-tuples, and let $S' \sset S$ be the subset consisting of $k$-tuples of weight strictly greater than $\al(\sum_{i=\cl{r}}^n i)$. It is enough to show that $S'$ is empty. Suppose, for the sake of contradiction that this is not the case. Let $\tau \in S'$ can be chosen so that 
$$
J(\tau) := \max \{j \in \{1, \dots, k\} : \tau_j \ne \pi_j\}
$$
is minimal among all paths in $S^\prime$, and such that for any $\sig \in S'$ with $J(\sig) = J(\tau)$, the path $\sig_{J(\tau)}$ is \emph{not} to the right of $\tau_{J(\tau)}$. For ease of notation, set $j = J(\tau)$. By this choice, note that $\tau_j$ must be to the left of $\pi_j$. We will show that we can push $\tau_j$ further to the right while remaining in $S^\prime$, which will be a contradiction.

Since $\tau_j$ has jumps at integer times and is to the left of $\pi_j$, it necessarily collects an atom at some location $(t +\ell, \ell)$ for some $\ell < n - k + j$. It also collects an atom at some location $q:=(t+ \ell + 1, m)$ for some maximal index $m$. The restriction of the path $\tau_j$ from $(t + r, n)$ to $q$ then has weight
\begin{equation}
\label{E:tau-length}
\begin{split}
L := &-[g_m(n-\ell + m) - g_m(n - \ell + m - 1)]  \\ &+\sum_{i=1}^{m-1}[g_i(n - \ell + m - 1) - g_i(n - \ell + m - 2)] + \al(\ell - m + 2).
\end{split}
\end{equation}
Now, consider an alternate path $\sig_j$ which is equal to $\tau_j$ from $q$ to $(t +n, 1)$, but from $(t + r, n)$ to $q$ is given by the rightmost path. The path $\sig_j$ is to the right of $\tau_j$ but still to the left of $\pi_j$. Therefore setting $\sig_i = \tau_i$ for all $i \ne j$, $\sig$ is a disjoint $k$-tuple from $(t +r, n)^k$ to $(t+n, 1)$. 

Moreover, the length of $\sig_j$ from $(t + r, n)$ to $q$ is simply the sum of all the atoms in the vertical strip from $(t +\ell + 1, \ell + 1)$ to $(t +\ell + 1, m)$. This length $L'$ equals the second line in \eqref{E:tau-length}. Therefore by the first inequality in \eqref{E:GT-inequalities}, $L' \ge L$. Finally, by construction none of the paths $\tau_j, j \ne i$ can pick up any atoms in the vertical strip from $(t +\ell + 1, \ell + 1)$ to $(t +\ell + 1, m)$, so
$$
|\sig|_{\Delta_{t, \al} g} \ge |\tau|_{\Delta_{t, \al} g}
$$
and hence $\sig \in S'$. Also, $J(\sig) = J(\tau)$, and $\sig_j$ is to the right of $\tau_j$. This contradicts the choice of $\tau$, so $S'$ must be empty, as desired.  
\end{proof}


%

\begin{remark}
\label{R:iterated-RSK}
Just as we can build the melon map $W$ as a composition of $2$-line Pitman transforms, see Definition \ref{D:iterated-melon-def}, we can also build the $n$-line RSK correspondence by composing $\binom{n}{2}$ $2$-line correspondences. More precisely, we build up the melon $Wf$ using the maps $\sig_i f$ as in \eqref{eq:w}, but every time we apply one of the maps $\sig_i$ to an intermediate function $g$, we also record the additional value $g_{i+1}(t)$. 

The map $(g_i,g_{i+1}) \mapsto (\sig_i g_i, \sig_i g_{i+1}; g_{i+1}(t))$ is a $2$-line $\operatorname{RSK}_t$ correspondence and hence is invertible, therefore so is the whole correspondence. Moreover, the $\binom{n}{2}$ additional values that we record with this procedure correspond to the $\binom{n}{2}$ entries in the Gelfand-Tsetlin pattern $G_t g_i(j), 1 \le i \le j \le n-1$ that cannot be read off of $W_tf$.
 
Though this basic idea is fairly simple, we found that the method we chose present is more straightforward and geometrically intuitive.
\end{remark}

\begin{remark}
\label{R:other-weights}
Our $\operatorname{RSK}_t$ map is based on one method of embedding $\sD^n_t$ into $\sD^n_{*t}$ by adding heavy weights after time $t$. There are clearly many ways to do this, and different methods will result in different bijections. One common feature of these bijections is that the key data that they see about $f$ beyond its melon $Wf$ will be a collection of left-to-right last passage values from time $0$ to time $t$. Though we will not prove it here, all left-to-right last passage values are contained in $G_t f$, just as all bottom-to-top last passage values are contained in $W_tf$ by Proposition \ref{P:W=w} (i).

Another option for constructing an RSK-like bijection would be to add heavy weights \emph{before} time $0$, essentially embedding $f$ as an element of $\sD^n_\uparrow$. We could also add weights to both sides of $[0, t]$ to embed $f$ as a different distinguished element of an isometry class.

Bijections related to RSK exploring the use of different left-to-right or bottom-to-top last passage values have been constructed in \cite*{dauvergne2020hidden} and \cite*{garver2018minuscule}.
\end{remark}
		
	\subsection{Bijectivity for lattice specializations and other restrictions}
	
	Bijectivity of the cadlag RSK correspondence $\operatorname{RSK}_t:\sD^n_t \to \sG^n_t$ naturally implies that for any subset $A \sset \sD^n_t$, that $\operatorname{RSK}_t$ is also a bijection from $A$ to $\operatorname{RSK}_t(A)$. For certain subsets $A$, we can explicitly identify $\operatorname{RSK}_t(A)$, allowing us to recover previously known bijections and identify some new ones. In the next set of examples, we gather together the restricted bijections that correspond to classical integrable models of last passage percolation. In the $t = \infty$ setting where the Gelfand-Tsetlin pattern is dropped, these examples correspond exactly to those introduced immediately after Theorem \ref{T:Bernoulli-map-intro}.
	
	For these examples, we say that a cadlag function $f$ with positive jumps is \textbf{pure-jump} if $df$ is an atomic measure.
	
	\begin{example}
	\label{Ex:restrictions}
Let $t > 0$.
	\begin{enumerate}
		\item \textbf{Continuous functions.} If $A$ is the set of continuous functions $f \in \sD^n_t$ then $\operatorname{RSK}_t(A)$ is the set of pairs $(w, g) \in \sG^n_t$ such that $w$ is also in $A$ (e.g. $w$ is continuous as well). This setting of continuous functions is studied in detail in \cite*{biane2005littelmann}. 
		\item \textbf{Unit jumps.} Let $A$ be the set of pure-jump functions $f\in \sD^n_t$, such that every jump of each $f_i$ has size $1$, and such that all jumps of $f_i, f_j$ are at distinct locations for $i \ne j$.
		The $\operatorname{RSK}_t(A)$ is the set of pairs $(w, g) \in \sG^n_t$, where $w$ is also in $A$, and all entries of $g$ are nonnegative integers. In this setting, the space $\operatorname{RSK}_t(A)$ is equivalent to the decorated Young Tableau defined in \cite{nica2015decorated}.
		
		\item \textbf{Real jumps at integer times.} Let $A$ be the set of pure-jump functions $f\in \sD^n_t$, such that every $f_i$ only jumps at integer times. Then $\operatorname{RSK}_t(A)$ is the set of pairs $(w, g) \in \sG^n_t$, where $w$ is also in $A$.
		
		\item \textbf{Integer jumps at integer times.} Let $A$ be the set of pure-jump functions $f\in \sD^n_t$, such that every $f_i$ only jumps at integer times and all jumps have integer values. Then $\operatorname{RSK}_t(A)$ is the set of pairs $(w, g) \in \sG^n_t$, where $w$ is also in $A$, and all coordinates of $g$ are nonnegative integers. 
		
		
		\item \textbf{Bernoulli paths.}  Suppose that additionally, $t \in \N$. Let $A$ be the set of all functions $f \in \sD^n_t$ that are linear with slope in $\{0, 1\}$ on every integer interval $[i, i+1]$. Then $\operatorname{RSK}_t(A)$ is the set of pairs $(w, g) \in \sG^n_t$, where $w$ is also in $A$, and all coordinates of $g$ are nonnegative integers.
	\end{enumerate}
	\end{example}

It is easy to verify each of the five examples above from the explicit formulas for $\operatorname{RSK}_t$ and $\operatorname{RSK}_t^{-1}$ in Section \ref{SS:finite-bijection}. 	Example $4$ above corresponds to the usual RSK correspondence via Corollary \ref{C:melon-cor} and Example $5$ corresponds to the dual RSK correspondence via Corollary \ref{C:melon-cordual}.
	
\section{Preservation of uniform measure}
\label{S:preservation}

In each of the five examples in Example \ref{Ex:restrictions}, there are natural measures on $A$ that push forward tractable measures on $\operatorname{RSK}_t(A)$. By taking a limit as $t \to \infty$, we can also get tractable pushforward measures under the original melon map $W:\sD^n \to \sD^n$. Each of these measures corresponds to a classical integrable model of last passage percolation. This is summarized in the following table, essentially repeated from the introduction.

\[\begin{array}{c|l|l} \text{Example}	&  \text{Measure on $\sD^n$}  & \text{LPP model} \\ \hline
	\ref{Ex:restrictions}.1 & \text{Independent Brownian motions} &\text{Brownian LPP} \\
	\ref{Ex:restrictions}.2 & \text{Independent Poisson counting processes}  &\text{Poisson lines LPP} \\
	\ref{Ex:restrictions}.3 & \text{Independent discrete-time geometric random walks}  &\text{Geometric LPP} \\
	\ref{Ex:restrictions}.4 & \text{Independent discrete-time exponential random walks} &\text{Exponential LPP} \\
	\ref{Ex:restrictions}.5 & \text{Piecewise linear walks with independent Bernoulli slopes} &\text{S-J model}
	\end{array}
	\]
In all five examples above, the pushforward of these measures under $W$ is the nonintersecting version of these objects. This is known in all cases, e.g. see \cite*{o2003conditioned} and references therein, or Section 6 of \cite*{DNV1}. The standard proofs of these facts require explicit computations involving determinants and Doob transforms. 
Here we give an alternate approach that is computation-free. We demonstrate this in the case of Bernoulli walks, Example \ref{Ex:restrictions}.5.
	
We start with a more precise setup.
For $k \in \N \cup \{\infty\}$, a function $f:[0, k] \to \R$ is a \textbf{Bernoulli path} if on every integer interval $[n, n+1]$, $f$ is linear with slope in $\{0, 1\}$. A \textbf{Bernoulli walk} of drift $d \in [0, 1]$ is a random Bernoulli path whose slopes are independent Bernoulli random variables of mean $d$, and an \textbf{$n$-dimensional Bernoulli walk} with drift vector $d = (d_1, \dots, d_n)$ is an element of $\sD^n$ whose components are independent Bernoulli walks of drift $d_i$. See Figure \ref{fig:rskbig} for an illustration.

Now, for $t \in \N$ and an ordered vector $x=(x_1 \ge \dots \ge x_n) \in \{0, 1, \dots\}^n$, let $\nu_t(x)$ denote the uniform measure on $n$-tuples of ordered Bernoulli paths $f:[0, t] \to \R, f = (f_1 \ge \dots \ge f_n)$ that satisfy $f(0) = 0, f(t) = x$. There are only finitely many such $k$-tuples, so uniform measure is well-defined. 
A measure $\mu$ on the space of $n$-tuples of ordered Bernoulli paths 
$$
X = (X_1 \ge X_2 \ge \dots \ge X_n), \qquad X_i:[0, \infty) \to \R
$$
is a \textbf{Bernoulli Gibbs measure} if for any integer $t > 0$, the conditional distribution under $\mu$ of $X|_{[0, t]}$ given $X|_{[t, \infty)}$ is $\nu_t(X(t))$. 
We start by showing that $W$ maps Bernoulli walks to Bernoulli Gibbs measures.
\begin{figure}[t]
\centering
\includegraphics[width=0.7\textwidth]{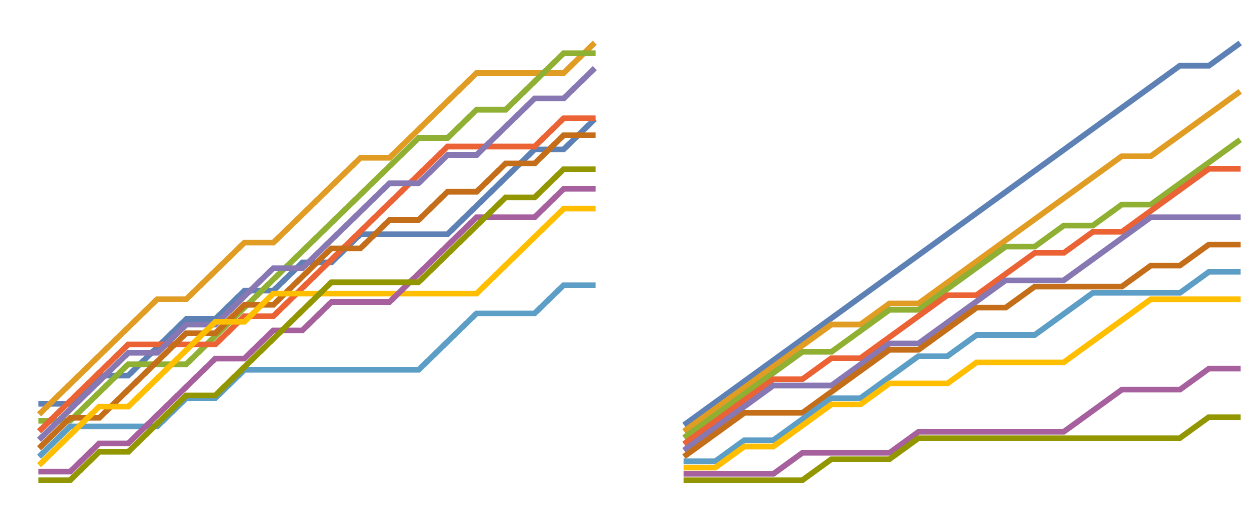}
\caption{10 Bernoulli walks $f$ and their melon $Wf$, 20 steps}
\label{fig:rskbig}
\end{figure}
\begin{theorem}
\label{T:Bernoulli-map}
Let $Y \in \sD^n$ be a Bernoulli walk of drift $d$. Then the law of $WY \in \sD^n_\uparrow$ is a Bernoulli Gibbs measure satisfying
\begin{equation}
\label{E:Gibbs-LLN-pre}
\lim_{t \to \infty} WY(t)/t = d^\circ
\end{equation}
almost surely, where $d^\circ = (d^\circ_1 \ge \dots \ge d^\circ_n)$ are the order statistics of $d$.
\end{theorem}

\begin{proof} Fix $t > 0$ and let $A, \operatorname{RSK}_t(A)$ be as in Example \ref{Ex:restrictions}.5. The map RSK$_t$ applied to $Y$ up to time $t \in \N$ gives an $n$-tuple of ordered paths $W_tY$ and a Gelfand-Tsetlin pattern $G_tY$. We first show that the law of $W_t Y$ given $G_t Y$ is $\nu_t(WY(t))$, and then use this to deduce the Bernoulli Gibbs property.
We first consider the case $d_i = 1/2$ for all $i$, so that the law of $Y|_{[0, t]}$ is the uniform measure on $A$.

By the bijectivity of $\operatorname{RSK}_t$ in Example \ref{Ex:restrictions}.5, the law of $\operatorname{RSK}_t Y$ is uniform on $\operatorname{RSK}_t(A)$. Therefore, conditionally on $G_tY$, which determines $WY(t)$, the law of $W_t Y$ is $\nu_t(WY(t))$. As an aside, the conditional law of $G_tY$ given $W_t Y$ is also independent and uniform on Gelfand-Tsetlin patterns with $n$th row $WY(t)$. 

Now for general $d\in [0,1]^n$, the law of $Y$ up to time $t \in \N$ is the uniform measure on $A$ biased by the Radon-Nikodym derivative 
$$
2^{nt} \prod_{i=1}^n d_i^{Y_i(t)} (1- d_i)^{t - Y_i(t)}.
$$
Since this derivative only depends on $Y(t)$, the  conditionally law of $Y$ given $Y(t)$ does not depend on the original drift vector $d$. Moreover, for any $i \in \{1, \dots, n\}$ we have
$$
Y_1(t)+\ldots +Y_i(t)=\sum_{j=1}^i G_tY_j(i),
$$
and so $Y(t)$ can be expressed from the the Gelfand-Tsetlin pattern $G_tY$. Therefore conditionally on $G_tY$, the law of $W_t Y = WY|_{[0, t]}$ does not depend on the drift $d$. Therefore as in the $d_i = 1/2$ case, the conditional law of $WY|_{[0, t]}$ given $G_tY$ is still $\nu_t(WY(t))$.

We now use this conditional law to prove the Bernoulli Gibbs property. First, this conditional law implies the stronger claim that for any integers $s \le t$, the conditional law of $WY|_{[0, s]}$ given $WY|_{[s, t]}$ and $G_tY$ is still $\nu_t(WY(s))$. Therefore it suffices to show that as $t \to \infty$, that $WY|_{[0, s]}$ and $G_tY$ are asymptotically independent. For this, it is enough to show that for any $k, \ell \in \{1, \dots, n\}$ with $\ell \ge n - k + 1$, for all large enough $t$ we have
\begin{equation}
\label{E:Ysum}
Y[(0, n)^k \to (\ell, t)^k] = Y[(s, n)^k \to (\ell, t)^k] + \sum_{i= n - k + 1}^n Y_i(s).
\end{equation}
Indeed, the right side of \eqref{E:Ysum} only depends on $Y|_{[s, \infty)} - Y(s)$ which is independent of $WY|_{[0, s]}$, and $G_tY$ can be expressed from the left hand side by varying $\ell, k$. Equation \eqref{E:Ysum} is equivalent to the claim that for large enough $t$, the rightmost disjoint optimizer from $(0, n)^k$ to $(\ell, t)^k$ follows the bottom $k$ paths up to time $s$. This follows from Remark \ref{R:sourness-of-random-walks}.


We now show that $WY$ satisfies \eqref{E:Gibbs-LLN-pre}. Define operators $\Lambda_k:\sD^n\to \sD^n$ by $\Lambda_k f(t) = f(kt)/k$. By the law of large numbers, as $k\to \infty, \Lambda_k Y(t)\to y(t):=dt$ uniformly on compact sets. Since $W$ is continuous with respect to the uniform-on-compact topology and commutes with $\Lambda_k$ by definition of last passage, we have
$$
\lim_{k \to \infty} \Lambda_k WY(1)=\lim_{k \to \infty} W\Lambda_k Y (1)=Wy(1).
$$ 
Finally, $W$ applied to linear functions just sorts them, so $Wy(t)=d^\circ t$.
\end{proof}

Next, we show that there is a unique Bernoulli Gibbs measure satisfying \eqref{E:Gibbs-LLN-pre} for every possible $d^\circ$.
\begin{theorem}
\label{T:Gibbs-LLN}
For any $d\in [0,1]^n$ with $d_1 \ge \dots \ge d_n$ there is a unique Gibbs measure $\mu_d$ on ordered $n$-tuples of Bernoulli paths in $\sD^n$ so that for $X \sim \mu_d$,
\begin{equation}\label{E:Gibbs-LLN}
\lim_{t \to \infty} WY(t)/t = d \qquad \text{a.s.}
\end{equation}
\end{theorem}

\begin{proof} 
Let $\mu_d$ denote the law of $WY^d$, where $Y^d$ is a Bernoulli walk of drift $d = (d_1 \le \dots \le d_n)$.

Now let $X$ be a sample from an arbitrary Bernoulli Gibbs measure satisfying \eqref{E:Gibbs-LLN}. To show that $X\sim \mu_d$, it suffices check that $X|_{[0, s]} \eqd WY^d|_{[0, s]}$ for all $s \in \N$.
Let $\eps>0$ and $v=(n,n-1,\ldots,1)$ and set
$$
\bar d_i=(d_i+\eps v_i)\wedge 1,\qquad \underline d_i=(d_i-\eps v_i)\vee 0, \qquad\bar X\sim \mathcal \mu_{\bar d}, \qquad {\underline X}\sim \mathcal 
\mu_{\underline d}.
$$ 
When $\underline d_i<d_i$, by \eqref{E:Gibbs-LLN}, for all $i$ we have
$\p(\underline X_i(t)\le X_i(t))\to 1$ as $t\to\infty$. Otherwise $d_i=\underline d_i=0$, but in this case $\underline X_i=0$ as well, so $\underline X_i\le X_i$ a.s.  Thus, after a symmetric upper bound, we get that
\begin{equation}\label{E:order}
\p A_t \to 1 \qquad \text{as } t \to \infty, \qquad \text{ where } A_t = \{\underline X(t)\le X(t)\le \overline X(t)\}.
\end{equation}
Now, the proof of Lemmas 2.6/2.7 in \cite*{CH} shows that if $x \le x'$ coordinatewise, $t \in \N$, and $Y \sim \nu_t(x), Y' \sim \nu_t(x')$ then we have the stochastic dominance $Y \preceq Y'$, so there exists a coupling with $Y_i(t) \le Y'_i(t)$ for all $i, t$. Thus, for $t \in \N$, on the event $A_s$ for $s \ge t$ we have
\begin{equation}
\label{E:At-have}
\underline X|_{[0, t]} \preceq X|_{[0, t]} \preceq \overline X|_{[0, t]}.
\end{equation}
As $s\to\infty$, since $\p A_s \to 1$ this implies \eqref{E:At-have} unconditionally.
Now let $\eps\to 0$. The laws $\mu_d$ restricted to $[0,t]$ are continuous in $d$ in the total variation norm, since the laws of $Y^d|_{[0, t]}$ are themselves continuous in $d$ in total variation. Therefore \eqref{E:At-have} holds even when $\eps=0$. In this case $\underline X$ and $\overline X$ both have distribution $\mu_d$ restricted to $[0, t]$, and hence so does $X|_{[0, t]}$, as required. 
\end{proof}

Theorems \ref{T:Bernoulli-map} and \ref{T:Gibbs-LLN} and Proposition \ref{P:W=w}(i) yield the following immediate corollary.
\begin{corollary}[Metric Burke property]\label{c:metric_burke}
Last passage percolation across an $n$-dimensional Bernoulli walk ignores the order of the drift vector. More precisely, if $Y, Z \in \sD^n$ are Bernoulli walks with drifts $d, e$ satisfying $d^\circ = e^\circ$, then $WY\eqd WZ$, and as functions of $x\le y$ we have
$$
Y[(x,n)\to (y,1)]\;\eqd\;Z[(x,n)\to (y,1)].
$$
\end{corollary}

Burke's theorem normally refers to a certain invariance between arrivals and departures in a queuing processes; Corollary \ref{c:metric_burke} is a kind of Burke property because it shows an invariance in the last passage value under exchanging the rows of the underlying environment. See \cite{o2002representation}.

The proof framework in this section goes through essentially verbatim if we start with a vector of independent geometric random walks, as in Example \ref{Ex:restrictions}.4. In this setting, the random walks $Y$ are embedded in $\sD^n$ as pure-jump paths with jumps at integer times. The Pitman ordering condition on $WY$ means that the output is a vector of geometric walks conditioned so that $WY_i(t) \ge WY_{i+1}(t + 1)$ for all $i, t$. Measure-preservation for the remaining three examples in Example \ref{Ex:restrictions} can be deduced by standard limiting procedures. We leave the details of this to the interested reader.

\section{Embedding classical versions of RSK}
\label{S:embedding}
	
	In this section, we relate our RSK map $\operatorname{RSK}_t$ to the usual RSK and dual RSK correspondences for nonnegative matrices. These correspondences are connected to last passage percolation in the lattice $\Z^2$. We start with the connection to the standard RSK correspondence.
	
	\subsection{Young tableaux}
	\label{S:Young-tableaux}
	We recall some basic combinatorial objects, see e.g. \cite{stanley2} for a detailed reference. A \textbf{partition} is a weakly decreasing sequence $\la = (\la_1, \la_2, \dots \la_k)$ of positive integers. The \textbf{size} of the partition is $|\la| = \sum_{i=1}^k \la_i$. To any partition $\la$, the \textbf{Young diagram} associated to $\la$ is the set of squares $Y(\la) = \{(i, j) \in \Z^2 : 1 \le i \le \la_j \}$. 
	A \textbf{semistandard Young tableau} of shape $\la$ is a filling of the corresponding Young diagram with positive integers such that the entries are strictly increasing along columns and weakly increasing along rows. 
	
	There is a natural correspondence between Young tableaux and Pitman ordered cadlag paths with only integer-valued positive jumps at positive integer times as in Example \ref{Ex:restrictions}.4. Consider a Young tableau $T$ of shape $\la = (\la_1, \la_2, \dots \la_k)$. Define $w \in \sD_\uparrow^k$ by setting
	$$
	w_i(t) = \text{$\#$ of entries in row $i$ of $T$ that are $\le t$}.
	$$
	In other words, the path $w_i$ has jumps precisely at the times $t$ which are equal to the entries of the $i$-th row of $T$. It is straightforward to check that with this definition, each $w_i$ is a cadlag path with positive integer jumps at integer times. The fact that the entries of $T$ are  strictly increasing along columns implies that $w_i(t^-) \ge w_{i+1}(t)$ for all $i, t$, and so the $w \in \sD^k_\uparrow$. This map from Young tableaux to Pitman ordered paths on this space is invertible. Moreover, for $n > k$ we can extend the collection $(w_1, \dots, w_k)$ to a collection $(w_1, \dots, w_n)$ of Pitman ordered paths by setting $w_i = 0$ for $i > k$.
	
	There is also a well-known correspondence between Young tableaux and Gelfand-Tsetlin patterns with nonnegative integer entries. Namely, for a Young tableau $T$ of shape $\la$ whose largest entry is less than or equal to $m$, define a Gelfand-Tsetlin pattern $g = \{g_i(j) : i \le j \le m\}$ by setting $g_i(j)$ to be equal to the number of entries in row $i$ of $T$ that are less than or equal to $j$. 

\subsection{Classical RSK via Greene's theorem} 	
	The RSK correspondence is a map between the space of nonnegative matrices with integer entries and pairs of semistandard Young tableaux $(P, Q)$ of equal shape. Typically it is described using a local bumping algorithm. However, the RSK bijection can alternately be described using last passage percolation. For the restriction of RSK to permutation matrices (the Robinson-Schensted correspondence) this is due to \cite*{greene1974extension}. A version of Greene's theorem for RSK is also well-known, but appears to be folklore and we do not know of an original reference. See, for example, Theorem 24 in \cite*{hopkins2014rsk} or \cite*{krattenthaler2006growth}, Theorem 8.
	
In the following, we describe RSK based on this connection with Greene's theorem in the language of last passage values. For two points $p = (x, n), q = (y, m) \in \Z^2$ with $x \le y$ and $n \ge m$, we say that a sequence of vertices $\pi = (\pi_1 = p, \dots, \pi_k = q)$ is a \textbf{directed path from $p$ to $q$} if $\pi_i - \pi_{i-1} \in \{(1, 0), (0, -1)\}$ for all $i$. For an array $A = \{A_u : u \in \Z^2\}$ of nonnegative numbers, we can define the weight of any path $\pi$ from $p$ to $q$ by
	\begin{equation}
	\label{E:lattice-weight}
	|\pi|_A= \sum_{v\in \pi} A_v.
	\end{equation}
	We also define the last passage value 
\begin{equation}
\label{E:LPP-lat}
	A[p\to q]= \max_{\pi} |\pi|_A,
	\end{equation}
	where the maximum is taken over all possible paths $\pi$ from $p$ to $q$. More generally, for vectors $\bp = (p_1, \dots, p_k), \bq = (q_1, \dots, q_k)$, define the {\bf multi-point last passage value}
	\begin{equation}
	\label{E:LPP-k}
	A[\bp \to \bq] =
	\max_{\pi_1,\ldots, \pi_k}|\pi_1|_A+\cdots +|\pi_k|_A
	\end{equation}
	where the maximum now is taken over all possible $k$-tuples of \emph{disjoint} paths, where each $\pi_i$ is a path from $p_i$ to $q_i$. This is defined so long as a disjoint $k$-tuple exists. We also introduce the shorthand $A[p^{*k} \to q^{*k}]$ for the $k$-point last passage value from
	$$
	(p - (0, k-1), \dots, p - (0, 1), p) \to (q, q + (1, 0), \dots, q + (k-1, 0)).
	$$
	The value $A[p^{*k} \to q^{*k}]$ is best thought of as a last passage value with $k$ disjoint paths from $p$ to $q$, hence the similar notation to the corresponding object in the cadlag setting. We are forced to stagger the start and end points of the paths to allow for disjointness.
	
	Now for an $n \X m$ matrix of nonnegative integers $A$ (equivalently, a restriction of a nonnegative array to the set $\{1, \dots, m\} \X \{1, \dots, n\}$), we can define a semistandard Young tableau, called the {\bf recording tableau} $Q$  with at most $n \wedge m$ rows and entries in $\{1, \dots, m\}$ by letting
	$$
	A[(1, n)^{*k \wedge i} \to (i, 1)^{*k \wedge i}] = \text{$\#$ of entries in rows $1, \dots, k$ of $Q$ that are $\le i$}.
	$$ 
	Similarly, define a semistandard Young tableau, called the {\bf insertion tableau} $P$ with at most  $n \wedge m$ rows and entries in $\{1, \dots, n\}$ by letting
	$$
	A[(1, n)^{*k \wedge i} \to (m, n - i + 1)^{*k \wedge i}] = \text{$\#$ of entries in rows $1, \dots, k$ of $P$ that are $\le i$}.
	$$ 
	The RSK correspondence is the map $A \mapsto (Q, P)$. Observe that with these definitions $P$ and $Q$ have the same shape determined by the last passage values $A[(1, n)^{*k} \to (m, 1)^{*k}], k = 1, \dots, m \wedge n$. By the correspondences between semistandard Young tableaux and Pitman ordered collections of cadlag paths and Gelfand-Tsetlin patterns we can associate to $(Q, P)$ a pair $(W A, G A) \in \sG^n_m$. Unravelling the bijections in Section \ref{S:Young-tableaux}, we get that for all $t \in [1, m]$ and $k \in \{1, \dots, n\}$, we have
	\begin{equation}
	\label{E:j1kkW}
	\sum_{j=1}^k W A_j(t) = A[(1, n)^{*k \wedge \fl{t}} \to (\fl{t}, 1)^{*k \wedge \fl{t}}],
	\end{equation}
	and for $t < 1$, we have $WA_j(t) = 0$. Also, for $1 \le k \le i \le n$ we have
	\begin{equation}
	\label{E:j1kkM}
	\sum_{j=1}^k G A_j(i) = A[(1, n)^{*k} \to (m, n - i + 1)^{*k}].
	\end{equation}
	\subsection{Classical RSK and the melon map}
	\label{SS:latticecadlag}
	
	
	For a nonnegative $n \X m$ matrix $A$, define $f^A \in \sD^n_m$ by
	\begin{equation}
	\label{E:fM-def}
	f^A_k(0^-) = 0, \qquad \mathand \qquad f^A_k(t) - f^A_k(s) = \sum_{r \in (s, t]} A_{r, k}.
	\end{equation}
	We will show that discrete last passage values across $A$ equal last passage values across $f^A$.
	\begin{prop}
		\label{P:fG}
		For all tuples of points $\bp, \bq$ such that $A[\mathbf{p}\to \mathbf{q}]$ is defined, we have
		\begin{equation}
		\label{E:fG2}
		A[\mathbf{p}\to \mathbf{q}] = f^A[\mathbf{p}\to \mathbf{q}],
		\end{equation}
	\end{prop}
To prove Proposition \ref{P:fG}, we will show that lattice last passage values can be equivalently defined using unions of possibly overlapping paths.
	
We first prove this for endpoints that lie in a packed staircase configuration.
	
	\begin{lemma}
		\label{L:staircase}
		Let $\bp, \bq$ be such that $p_i = p_{i-1} + (1, 1)$ and $q_i = q_{i-1} + (1, 1)$ for all $i$. Then
		\begin{equation}
		\label{E:Gbpq}
		A[\bp \to \bq] = \max_{\pi_1, \dots, \pi_k} \sum_{v \in \bigcup_{i} \pi_i} A_v, 
		\end{equation}
		where the maximum is over all $k$-tuples of paths $\pi_i$ from $p_i$ to $q_i$, without any disjointness condition enforced. In the union in \eqref{E:Gbpq}, weights on multiple paths are only counted once.
	\end{lemma}

\begin{figure}[ht]
\centering
\includegraphics[width=0.5\textwidth]{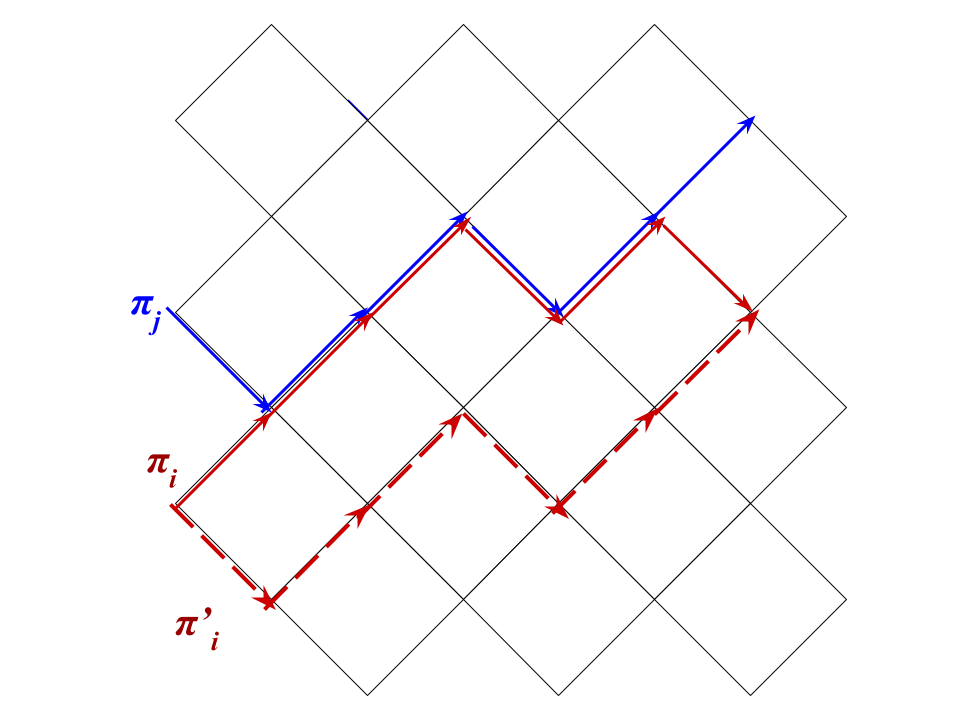}
\caption{Example of how the overlapping paths in the proof of Lemma \ref{L:staircase} are moved down to find a better configuration of non-intersecting paths.}
\label{fig:paths}
\end{figure}

For the proof, it will be easier to imagine the  coordinate system as rotated clockwise by $45$ degrees, and scaled up by $\sqrt{2}$ as in Figure \ref{fig:paths}. After this rotation, all the points $p_i$ lie on a common vertical line $\Z \X \{p^*\}$. Similarly, all the points $q_i$ lie on a common vertical line $\Z \X \{q^*\}$. Moreover, with this rotation any path $\pi$ from $p_i$ to $q_i$ for some $i$ gets transformed to the graph of a function $\hat \pi:\{p^*, \dots, q^*\} \to \Z$ with steps of $\pm 1$. That is, $\pi$ gets transformed to a simple random walk path $\hat \pi$.
	
	
	\begin{proof}
		The fact that $\operatorname{LHS} \le \operatorname{RHS}$ in \eqref{E:Gbpq} follows since we are maximizing over a smaller set on the left. To achieve the opposite inequality, we just need to show that there is a set of disjoint paths $\pi_i$ that achieves the maximum on the right side of \eqref{E:Gbpq}. 
		Without loss of generality, by passing to order statistics, we may assume that the maximum is achieved on a $k$-tuple of paths $\tau$ satisfying
		\begin{equation}
		\label{E:tau-1}
		\hat \tau_1(x) \le \hat \tau_2(x) \dots \le \hat \tau_k(x),
		\end{equation}
		for all $x =p^*,p^*+1,\ldots , q^*$. 
		
		Now consider the set $S$ of all $k$-tuples which achieve the maximum in \eqref{E:Gbpq} and satisfy \eqref{E:tau-1}. We put a partial order on this set by saying that $\pi \le \tau$ if $\hat \pi_i(x) \le \hat \tau_i(x)$ for all $i \in \{1, \dots, k\}, x = p^*,p^*+1,\ldots , q^*$. Let $\pi$ be a minimal element of the finite set $S$.  We show that $\pi$ consists of disjoint paths.
		
%
%
		
		Suppose not. Then there exists an $i < j$ and a value $x$ such that $\hat \pi_i(x) = \hat \pi_j(x)$. We may also assume that $i$ is the minimal such index where there is such a conflict, and hence that
		\begin{equation}
		\label{E:hatpiell}
\hat \pi_i(x) - 2\ge \hat \pi_\ell(x) \qquad \text{ for all } i > \ell.
		\end{equation}
		 Let $I = \{a, \dots, b\} \sset \{p^*, \dots, q^*\}$ be the largest interval containing $x$ such that $\hat \pi_i = \hat \pi_j$ on $I$. Since the start and endpoints of $\hat \pi_i, \hat \pi_j$ are distinct, we have $p^* < a, b < q^*$. Therefore $\hat \pi_i, \hat \pi_j$ are well-defined at $a-1$ and $b+1$ and satisfy
		 $$
		 \hat \pi_i(a-1) = \hat \pi_j(a - 1) - 2, \qquad \hat \pi_i(b+1) = \hat \pi_j(b + 1) - 2.
		 $$ 
		 Therefore the function $\hat \pi_i'$ which is equal to $\hat \pi_i$ on $I^c$, and shifted down by 2 units, $\hat \pi^\prime_j = \hat \pi_j - 2$, on $I$ is also a simple random walk path, see Figure \ref{fig:paths}. Thus the $k$-tuple $\pi' = (\pi_1, \dots, \pi_i', \pi_{i+1}, \dots, \pi_k)$ also consists of paths from $p_i$ to $q_i$. Moreover, the vertices covered by $\pi'$ contain all the vertices covered by $\pi$, so because the weights are all non-negative, $\pi'$ must also achieve the maximum in \eqref{E:Gbpq}. Finally, by \eqref{E:hatpiell}, the $k$-tuple $\pi'$ still satisfies inequalities in \eqref{E:tau-1}, so $\pi' \in S$. On the other hand, $\pi' \le \pi$ by construction, contradicting the minimality of $\pi$.
	\end{proof}
	
	We can now extend this to general endpoints.
	
	\begin{lemma}
		\label{L:discrete-end}
	For any $\bp, \bq$ such that $A[\bp \to \bq]$ is defined, we have
		\begin{equation}
		\label{E:Gbpq2}
		A[\bp \to \bq] = \max_{\pi_1, \dots, \pi_k} \sum_{v \in \bigcup_{i} \pi_i} A_v, 
		\end{equation}
		where the maximum is over all $k$-tuples of paths $\pi_i$ from $p_i$ to $q_i$, without any disjointness condition enforced.
	\end{lemma}
	
	\begin{proof} 
		We can find a pair of vectors $(\bp', \bq')$ that are of the form in Lemma \ref{L:staircase} such that there are sets of disjoint paths $\pi$ from $\bp'$ to $\bp$ and $\tau$ from $\bq$ to $\bq'$. Let $H$ be a nonnegative array which is equal to $1$ for $x \in \cup \pi \cup \tau$, and zero otherwise, and let $A' = A + sH$. Then for large enough $s$, letting $m = |\cup \pi \cup \tau|$ we have
		\begin{equation}
		\label{E:G'}
		A'[\bp' \to \bq'] = ms + A[\bp \to \bq],
		\end{equation}
		since any optimal disjoint paths from $\bp'$ to $\bq'$ will necessarily follow $\pi$ and $\tau$. By Lemma \ref{L:staircase}, we similarly have that
		\begin{equation}
		\label{E:msR}
		A'[\bp' \to \bq'] = ms + R,
		\end{equation}
		where $R$ denotes the right hand side of \eqref{E:Gbpq2}. Equating \eqref{E:G'} and \eqref{E:msR} completes the proof. 
	\end{proof}
	
	\begin{proof}[Proof of Proposition \ref{P:fG}]
		Any disjoint lattice paths from $\bp$ to $\bq$ can be mapped to disjoint cadlag paths, so we have $A[\bp \to \bq] \le f^A[\bp \to \bq]$. Now let
		$$
		f^A\{\bp \to \bq\} = \max_\pi |\pi|_f,
		$$
		where the maximum is now over $k$-tuples from $\bp$ to $\bq$ with the disjointness condition removed. In $|\pi|_f$ we only count weights once even if they are covered by multiple paths. Let $\pi$ be a $k$-tuple that achieves this maximum, and define a new $k$-tuple $\fl{\pi}$ by setting $\fl{\pi}_i(t) = \pi_i(\fl{t})$ for all $i, t$. Since $f^A$ has only positive jumps and is constant between integer times, $\fl{\pi}$ also achieves this maximum. Each $\fl{\pi}_i$ corresponds to a discrete lattice path $\pi'_i$ from $p_i$ to $q_i$, and we have the equality
		$$
		|\fl{\pi}|_{f^A} = \sum_{v \in \bigcup_i \pi'_i} A_v.
		$$
		Therefore by Lemma \ref{L:discrete-end}, $f^A\{\bp \to \bq\} \le A[\bp \to \bq]$. Since $f^A[\bp \to \bq] \le f^A\{\bp \to \bq\}$, we have that $f^A[\bp \to \bq] \le A[\bp \to \bq]$ as well.
	\end{proof}
	
	Finally, we can show that the usual RSK bijection is a special case of the cadlag RSK bijection.
	
	\begin{corollary}
		\label{C:melon-cor}
		Let $A$ be an $n \X m$ matrix. Define $f^A \in \sD^n_m$ via the formula in \eqref{E:fM-def}. Then with $(WA, GA)$ as in \eqref{E:j1kkW} and \eqref{E:j1kkM} we have that
		$$
		(WA, GA) = (W_m f^A, G_m f^A) = \operatorname{RSK}_m(f^A).
		$$
	\end{corollary}
	
	\begin{proof}
		By Proposition \ref{P:fG} and tracing through the definitions, it suffices to show that 
		\begin{equation}
		\label{E:fA-eqn}
			f^A[(0, n)^k \to (t, j)^k] = f^A[(1, n)^{*k  \wedge 
		\fl{t}}  \to (\fl{t}, j)^{*k \wedge \fl{t}}]
		\end{equation}
		for all $t, j, k$ with $k \le n + 1- j$. For $k \ge \fl{t}$, both sides pick up all weights of $A$ in the box $\{1, \dots, \fl{t}\} \X \{1, \dots, n\}$. For $k < \fl{t}$, notice that since $f^A(t) = 0$ for all $t < 1$ and $f^A$ is unchanging between integer times, that 	$f^A[(0, n)^k \to (t, j)^k] = 	f^A[(1, n)^k \to (\fl{t}, j)^k]$. Moreover, essential disjointness at times $1$ and $\fl{t}$ implies that any disjoint $k$-tuple from $(1, n)^k$ to $(\fl{t}, j)^k$ has the same length as some disjoint $k$-tuple from $(1, n)^{*k}  \to (\fl{t}, j)^{*k}$.
	\end{proof}

\begin{remark}
	\label{R:cts-rsk}
	While the RSK correspondence is defined only for matrices with nonnegative integer entries, the maps \eqref{E:j1kkM} and \eqref{E:j1kkW} are still defined for matrices $A$ with nonnegative real entries; there is just no longer a connection with Young tableaux. Proposition \ref{P:fG}, Lemma \ref{L:staircase}, and Corollary \ref{C:melon-cor} still hold in this generality and the proofs go through verbatim.
\end{remark}
	
	\subsection{Dual RSK}
	
	The dual RSK correspondence can also be connected with $\operatorname{RSK}_t$ via lattice last passage. The necessary version of Greene's theorem for dual RSK is Theorem 10 in \cite*{krattenthaler2006growth}. As the details connecting cadlag RSK and dual RSK are similar to the case of the usual RSK correspondence, we only include theorem statements here.
	
	Let $A$ be an $n \X m$ matrix of $0$s and $1$s. For two points $p = (x, k), q = (y, \ell)$ with $x \le y$ and $k \ge \ell$, we say that $\pi = (\pi_1 = p, \dots, \pi_k = q)$ is a \textbf{dual path from $p$ to $q$} if $\pi_i - \pi_{i-1} \in \{(1, s) : s \in \Z_{\le 0} \}$ for all $i$. That is $\pi$ is a path that moves strictly to the right and weakly up at every step.
	 Definitions \eqref{E:lattice-weight}, \eqref{E:LPP-lat}, and \eqref{E:LPP-k} still make sense for dual paths and we write $A\{\bp \to \bq\}$ for a last passage value with dual paths.
	
	Now, for a filling $Q$ of a Young diagram $Y$, we write $Q^T$ for the transposed filling of the transposed Young diagram $Y^T$, i.e. a cell $(a, b) \in Y$ if and only if $(b, a) \in Y^T$ and $Q^T(b, a) = Q(a, b)$. For an $n \X m$ matrix $A$ of $0$s and $1$s, we define a semistandard Young tableau $P$ with at most $n \wedge m$ rows and entries in $\{1, \dots, m\}$ by letting
	$$
	A\{(1, n)^{*k \wedge i} \to (i, 1)^{*k \wedge i}\}= \text{$\#$ of entries in rows $1, \dots, k$ of $Q^T$ that are $\le i$}.
	$$ 
	Also define a semistandard Young tableau $P$ with at most  $n \wedge m$ rows and entries in $\{1, \dots, n\}$ by letting
	$$
	A\{(1, n)^{*k \wedge i} \to (m, n - i + 1)^{*k \wedge i}\} = \text{$\#$ of entries in rows $1, \dots, k$ of $P$ that are $\le i$}.
	$$ 
	The dual RSK correspondence is the map $A \mapsto (Q, P)$ which maps $0-1$ matrices to pairs of semistandard Young tableaux such that the shapes of $Q$ and $P$ are \textbf{conjugate}, i.e $Q^T$ has the same shape as $P$. Observe that with the above definitions $Q^T$ and $P$ have the same shape.
	
	 The fact that $Q$, rather than $Q^T$, is a semistandard Young tableau is a consequence of the differences in the definition of paths and dual paths. Nonetheless, to connect this definition to cadlag RSK it is still $Q^T$ that we want to write as a collection of Pitman ordered paths $(W A_1, \dots W A_n)$. To do this, we embed $Q^T$ not as a collection of cadlag paths with jumps, but rather as a collection of paths with piecewise linear increments. For all $t \in \{1, \dots m\}$ and $k \in \{1, \dots, n\}$, we write
	\begin{equation}
	\label{E:j1kkWdual}
	\sum_{j=1}^k W A_j(t) = A\{(1, n)^{*k \wedge t} \to (t, 1)^{*k \wedge t}\}.
	\end{equation}
	We also set $WA(0) = 0$, and let each line $WA_i$ be linear on every interval $[t, t + 1]$ with $t \in \Z$. Since $Q$ is a semistandard Young tableau, with this definition each line $WA_i$ either has slope $0$ or slope $1$ on every interval. We also turn $P$ into a Gelfand-Tsetlin pattern $G$ in the usual way. For $1 \le k \le i \le n$ we have
	\begin{equation}
	\label{E:j1kkMdual}
	\sum_{j=1}^k G A_j(i) = A\{(1, n)^{*k} \to (m, n - i + 1)^{*k}\}.
	\end{equation}
	We now connect this description to cadlag RSK. For an $n \X m$ $\{0, 1\}$-matrix $A$, define $\ell^A \in \sD^n_m$ by letting
	\begin{equation}
	\label{E:lM-def}
	\ell^A_k(0) = 0, \qquad \mathand \qquad \ell^A_k(t) - \ell^A_k(s) = \sum_{r \in [s+1, t]} A_{r, k} \qquad \text{ for } s, t \in \Z,
	\end{equation}
	and by letting each $\ell^A_k$ be linear between integers. We then have the following analogue of Proposition \ref{P:fG}.
		\begin{prop}
		\label{P:fGdual}
		Let $\bp, \bq$ be such that $A\{\mathbf{p}\to \mathbf{q}\}$ is defined. Then
		\begin{equation}
		\label{E:fG}
		A\{\mathbf{p}\to \mathbf{q}\} = \ell^A[\mathbf{p}\to \mathbf{q}],
		\end{equation}
	\end{prop}
	
	The proof of Proposition \ref{P:fGdual} is similar to the proof of Proposition \ref{P:fG}. 
Proposition \ref{P:fGdual} leads to an analogue of Corollary \ref{C:melon-cor}.
	
		\begin{corollary}
		\label{C:melon-cordual}
		Let $A$ be an $n \X m$ matrix of $0$s and $1$s. Then with $(WA, GA)$ as in \eqref{E:j1kkWdual} and \eqref{E:j1kkMdual} we have that
		$$
		(WA, GA) = (W_m \ell^A, G_m \ell^A) = \operatorname{RSK}_m(\ell^A).
		$$
	\end{corollary}

	\begin{remark}
		\label{R:cts-dual-rsk}
		While the dual RSK correspondence is defined only for matrices with $\{0,1\}$ entries, the maps \eqref{E:j1kkMdual} and \eqref{E:j1kkWdual} are still defined for matrices $A$ with arbitrary real entries; there is just no longer a connection with Young tableaux. Proposition \ref{P:fGdual}  and Corollary \ref{C:melon-cordual} still hold in this generality.
	\end{remark}














\appendix

\section{Appendix: technical proofs}
\label{S:appendix}
 \begin{proof}[Proof of Lemma \ref{L:wm-lem}]
		Set
		$$
		s(x, y) = \sup_{z \in [x, y]} f_2(z) - f_1(z^-) = \sup_{z \in [x, y]} s(z, z),
		$$
		so that we have 
		$$
		s(x, y^-) = \sup_{z \in [x, y)} f_2(z) - f_1(z^-).
		$$
		For each $x$, the function $s(x, \cdot)$ is increasing. Also, since the functions $f_i$ are cadlag with positive jumps, we have that $s(x, \cdot)$ is cadlag. (Note that this would not hold if we allowed negative jumps in $f_1$.) We also have
		\begin{equation}
		\label{E:s-identities}
		s(x, y) - s(x, y^-) = [s(y, y) - s(x, y^-)]^+ \le f_2(y) - f_2(y^-).
		\end{equation}
		We can explicitly write last passage values across $f$ as
		\begin{equation}
		\label{E:lpp-2-lines}
		f\big[(x, 2) \to (y, 1) \big] = f_1(y) - f_2(x^-) + s(x, y).  
		\end{equation}
		Specializing to the case $x = 0$, and using that $Wf_1+Wf_2=f_1+f_2$, we have 
		\begin{align}
		\label{E:Wf-formulas}
		Wf_1(t) = f_1(t) + s(0, t) \qquad \mathand \qquad Wf_2(t) = f_2(t) - s(0, t).
		\end{align}
		From the fact that $s(0, \cdot)$ is cadlag and increasing, the function $Wf_1$ is cadlag with only positive jumps. Also, by \eqref{E:s-identities}, the function $Wf_2$ is cadlag with only positive jumps, so $W$ maps $\sD^2$ to itself.
		The last passage value across $Wf$ is
		\begin{align*}
		Wf\big[(x, 2) \to (y, 1) \big] &= Wf_1(y) - Wf_2(x^-) + \sup_{z \in [x, y]} Wf_2(z)-Wf_1(z^-).
		\end{align*}
		Substituting the formulas \eqref{E:Wf-formulas} we get that this equals
		\begin{equation}
		\label{E:f1f2s}
		f_1(y) + s(0, y) - f_2(x^-) + s(0, x^-) + \sup_{z \in [x, y]} [s(z,z)- s(0, z)-s(0,z^-)].
		\end{equation}
		By comparing with \eqref{E:lpp-2-lines}, we can see that the lemma will follow from the equality
		\begin{equation}
		\label{E:sup-equal}
		s(x, y) - s(0, y) - s(0, x^-) = \sup_{z \in [x, y]} [s(z,z) - s(0, z)-s(0,z^-)].
		\end{equation}
		To prove \eqref{E:sup-equal}, we divide into cases. First suppose that $s(0, x^-) = s(0, y)$. In this case, since $s(0, \cdot)$ is nondecreasing, we have that $s(0, z) =s(0,z^-)= s(0, x^-) = s(0, y)$ for all $z \in [x, y]$. Therefore
		\begin{align*}
		\sup_{z \in [x, y]} [s(z,z) - s(0, z)-s(0,z^-)] &= \sup_{z \in [x, y]} s(z,z) - s(0, y) - s(0, x^-) \\
		&= s(x, y) - s(0, y)-s(0, x^-).
		\end{align*}
		We turn to the  case when $s(0, x^-) < s(0, y)$. By definition, 
		\begin{equation}
		\label{E:s-0-x}
		s(0,y)=s(0,x^-)\vee s(x,y), \quad \text{so} \quad s(0,y)=s(x,y).
		\end{equation}
		
		Set 
		$$
		z_0 = \sup\{z \in [x, y]\,: \,s(0, z^-) =s(0, x^-)\}. 
		$$
		The function $s(0, \cdot^-)$ is left continuous, so this is in fact a maximum. In particular, since $s$ is nondecreasing, for each $z_1>z_0$ 
		$$s(0,z_1)  \ge s(0,z_1^-) > s(0,z_0^-).$$
		So we have, by definition of $s$
		$$
		s(0,z_1)=s(0,z_0^-)\vee s(z_0,z_1)=s(z_0,z_1).
		$$
		By  the right continuity of $s(0, \cdot)$ and $s(z_0, \cdot)$, as $z_1\downarrow z_0 $ we get 
		$
		s(0,z_0)=s(z_0,z_0).$ By choosing $z=z_0$ in the supremum on the right hand side of \eqref{E:sup-equal} we get
		\begin{equation}
		\label{E:ineq-sup}
		-s(0,x^{-}) = -s(0,z_0^-)\le \sup_{z \in [x, y]} [s(z,z) - s(0, z) - s(0, z^-)] 
		\end{equation}
		Since $s(z,z)\le s(0,z)$, and the fact that $s$ is nondecreasing, the right hand side can be upper bounded by 
		$$
		\sup_{z \in [x, y]}[ -s(0, z^-) ] = - s(0, x^-).
		$$ 
		so \eqref{E:ineq-sup} is in fact an equality. Since $s(0, y) = s(x, y)$ by \eqref{E:s-0-x}, this proves the preservation of last passage values in \eqref{E:sup-equal}. 
\end{proof}

\begin{proof}[Proof of Lemma \ref{L:push-back}]
We use the notation from the proof of Lemma \ref{L:wm-lem} in this appendix. Recall that 		$$
		s(x, y) = \sup_{z \in [x, y]} f_2(z) - f_1(z^-) = \sup_{z \in [x, y]} s(z, z).
		$$
Let
$$
I_f(x, y) = \{z \in [x, y] : s(z, z) = s(x, y)\}.
$$
This is the set of all possible jump times from line $2$ to $1$ for geodesics from $(x, 2)$ to $(y, 1)$ in $f$.
Also set $r(x, y) = \sup_{z \in [x, y]} [s(z, z) - s(0, z) - s(0, z^-)]$ and let 
$$
I_{Wf}(x, y) = \{z \in [x, y] : r(z, z) = r(x, y)\}.
$$
By \eqref{E:f1f2s} this is the set of all possible jump times from line $2$ to $1$ for geodesics from $(x, 2)$ to $(y, 1)$ in $Wf$.
Then the desired ordering on geodesics holds if and only if
$$
\inf I_{Wf}(x, y) \le \inf I_f(x, y), \qquad \text{ and } \qquad \sup I_{Wf}(x, y) \le \sup I_f(x, y).
$$
Again, we first deal with the case when $s(0, x^-) = s(0, y)$. In this case, for all $w \in [x, y]$ we have
$$
r(w, w) = s(w, w) - 2 s(0, x^-), 
$$
so $I_f(x, y) = I_{Wf}(x, y)$. Now suppose $s(0, x^-) < s(0, y)$. Define 
$$
A = \sup \{z \in [x, y] : s(0, z^-) = s(0, x^-)\}, \qquad B = \inf \{z \in [x, y] : s(0, z) = s(0, y)\}.
$$
We clearly have $A \le B$. Moreover,
$s(0, y) = s(x, y)$ in this case, and so for $z < I$, we must have $s(z, z) < s(x, y)$. Hence
$I_f(x, y) \sset [B, y]$. To complete the proof, we show $I_{Wf}(x, y) \sset [x, A].$ Since \eqref{E:ineq-sup} is an equality in this case, see the discussion following that inequality, at every point $w \in I_{Wf}(x, y)$, we have $r(w, w) = - s(0, x^-)$. Moreover, for $w > A$, we have
\begin{align*}
r(w, w) &= s(w, w) - s(0, w) - s(0, w^-) \le - s(0, w^-)
< - s(0, x^-).
\end{align*}
Therefore $I_{Wf}(x, y) \sset [x, A].$ 
\end{proof}

\bigskip

\noindent {\bf Acknowledgments.}  D.D. and M.N. were supported by NSERC postdoctoral fellowships. B.V. was supported by the Canada Research Chair program, the NSERC Discovery Accelerator grant.

\bibliographystyle{dcu}

\bibliography{Citations}

\end{document}